\documentclass[11pt]{article} 

\usepackage{amsmath}
\usepackage{amssymb}
\usepackage{color}
\usepackage{enumerate}
\usepackage{theorem}
\usepackage[T1]{fontenc}
\usepackage[latin1]{inputenc}
\def\r{\rightarrow}

\newcommand{\fdem}{\hspace*{\fill}~$\Box$\par\endtrivlist\unskip}

\newcommand{\E}{\mathbb{E}}     
\renewcommand{\P}{\mathbb{P}}     
     
\newcommand{\F}{\mathbb{F}}     
     
\newcommand{\N}{\mathbb{N}}     
\newcommand{\Z}{\mathbb{Z}}
\newcommand{\R}{\mathbb{R}}     
     
\newcommand{\C}{\mathbb{C}} 
 
\newcommand{\X}{\mathbb{X}} 
\newcommand{\T}{\mathbb{T}}

\renewcommand{\r}{\mathop{\rightarrow}}
\newcommand{\cA}{\mbox{$\cal A$}}

\newcommand{\cC}{\mbox{$\cal C$}}
\newcommand{\cD}{\mbox{$\cal D$}}

\newcommand{\cF}{\mbox{$\cal F$}}
\newcommand{\cG}{\mbox{$\cal G$}}

\newcommand{\cM}{\mbox{$\cal M$}}
\newcommand{\cN}{\mbox{$\cal N$}}
\newcommand{\cP}{\mbox{$\cal P$}}

\newcommand{\cS}{\mbox{$\cal S$}}

\newcommand{\cY}{\mbox{$\cal Y$}}

\newcommand{\cU}{\mbox{$\cal U$}}

\newtheorem{theo}{Theorem}[section]
\newtheorem{pro}{Proposition}[section]
\newenvironment{proof}[1]{\textit{Proof#1.\,}}{\fdem}
\newtheorem{lem}{Lemma}[section]
\newtheorem{rem}{Remark}[section]

\newtheorem{ass}{AC\hspace*{-3pt}}
\newtheorem{assu}{U\hspace*{-3pt}}
\newtheorem{assu3}{U3\hspace*{2pt}-\hspace*{-4pt}}

{\theoremstyle{break}\theorembodyfont{\rmfamily}
}

\newcommand{\IP}{\textbf{(I-A)}} 
\newcommand{\M}[1]{\boldsymbol{(\mathrm{M}#1)}} 
\newcommand{\NL}{\textbf{(NL)}}

\begin{document}

\title{Additional material on local limit theorem for finite Additive Markov Processes \cite{HerLed13}}  

\author{Loïc HERV\'E and James LEDOUX  \footnote{
Université
Européenne de Bretagne, I.R.M.A.R. (UMR-CNRS 6625), Institut National des Sciences 
Appliquées de Rennes. Loic.Herve,JamesLedoux@insa-rennes.fr}}

\maketitle

\begin{abstract}
This paper proposes additional material to the main statements of \cite{HerLed13} which are are recalled in Section~\ref{sec-ass-stat}. In particular an application of \cite[Th~2.2]{HerLed13} to Renewal Markov Processes is provided in Section~\ref{MRP} and a detailed checking of the assumptions of \cite[Th~2.2]{HerLed13} for the joint distribution of local times of a finite jump process is reported in Section~\ref{appli-local-times}. A uniform version of \cite[Th~2.2]{HerLed13}  with respect to a compact set of transition matrices is given in Section~\ref{sec-unif} (see \cite[Remark 2.4]{HerLed13}). The basic material on the semigroup of Fourier matrices and the spectral approach used in \cite{HerLed13} is recalled in Section~\ref{proof-LLT} in order to obtain a good understanding of the properties involved in this uniform version. \\[1mm]
\end{abstract}

\vspace*{-8mm}
{\small
\tableofcontents
}

\section{Notations}

Any vector $v=(v_k)_{k\in \{1,\ldots,N\}}\in\C^N$ is a considered as a row-vector and $v^{\top}$ is the corresponding column-vector. The vector with all components equal to 1 is denoted by $\mathbf{1}$. The euclidean product scalar and its associated norm  on $\C^N$ is denoted by $\langle\cdot,\cdot\rangle$ and $\| \cdot\|$ respectively.  
The set of $N\times N$-matrices with complex entries is denoted by $\cM_N(\C)$. 
Let $\|\cdot\|_\infty$ denote the supremum norm on $\C^N$:  $\forall v\in\C^N,\  \|v\|_\infty=\max_{k\in\{1,\ldots,N\}} |v_k)|$. For the sake of simplicity, $\|\cdot\|_\infty$ also stands for the associated matrix norm: 
$$\forall A\in\cM_N(\C),\quad \|A\|_\infty := \sup_{\|v\|_\infty=1}\|A v^{\top}\|_\infty.$$
We also use the following norm $\|\cdot\|_0$ on $\cM_N(\C)$:  
$$\forall A=(A_{k,\ell})_{(k,\ell)\in\{1,\ldots,N\}^2}\in\cM_N(\C),\quad \|A\|_0:= \max_{(k,\ell)\in\{1,\ldots,N\}} |A_{k,\ell}|.$$ 
It is easily seen that 
\begin{equation} \label{equiv-norm}
\forall A\in\cM_N(\C),\quad \|A\|_0\leq \|A\|_\infty \leq N \|A\|_0.
\end{equation}

For any bounded positive measure $\nu$ on $\R^d$, we define its Fourier transform as: 
$$\forall \zeta\in\R^d,\quad \widehat\nu(\zeta) := \int_{\R^d}e^{i \langle \zeta , y \rangle}\, d\nu(y).$$
Let  $\cA=(\cA_{k,\ell})_{(k,\ell)\in\{1,\ldots,N\}^2}$ be a $N\times N$-matrix with entries in the set of bounded positive measures on $\R^d$. We set 
\begin{subequations} 
\begin{eqnarray}
\forall B\in B(\R^d), && \cA(1_B) := \big(\cA_{k,\ell}(1_B)\big)_{(k,\ell)\in\X^2} \\
\forall \zeta\in\R^d, && \widehat \cA(\zeta) := (\widehat\cA_{k,\ell}(\zeta))_{(k,\ell)\in\X^2}. \label{def-Four-mat}
\end{eqnarray}
\end{subequations}
We denote by $\lfloor t \rfloor$ the integer part of any $t\in\T$. 
\section{The LLT for the density process} \label{sec-ass-stat}
Let $\{(X_t,Z_t)\}_{t\in \T}$ be an MAP with state space $\X \times \R^d$, where  $\X:=\{1,\ldots,N\}$ and the driving Markov process $\{X_t\}_{t\in\T}$ has transition semi-group  $\{P_t\}_{t\in\T}$. We refer to \cite[Chap. XI]{Asm03} for the basic material on such MAPs. The conditional probability to $\{X_0=k\}$ and its associated expectation are denoted by $\P_k$ and $\E_k$ respectively. Note that if $T :\R^d\r\R^m$ is a linear transformation, then $\{X_t,T(Z_t)\}_{t\in\T}$ is still a MAP on $\X\times R^m$ (see Lemma~\ref{lem_TMAP}). We suppose that  $\{X_t\}_{t\in\T}$ has a unique invariant probability measure $\pi$. Set $m:= \E_{\pi}[Z_1]=\sum_k \pi(k) \E_{k}[Z_1]\in\R^d$. Consider the centered MAP $\{(X_t,Y_t)\}_{t\in \T}$ where $Y_t := Z_t - t\, m$. The two next assumptions are involved in both CLT and LLT. 

\noindent \IP ~: 
{\it The stochastic $N\times N$-matrix $P:=P_1$ is irreducible and aperiodic.}

\noindent $\M{\alpha}$ : 
{\it 
The family of r.v.~$\{Y_v\}_{v\in(0,1]\cap\T}$ satisfies the  uniform moment condition of order $\alpha$:} 
\begin{equation} \label{moment-alpha} 
M_{\alpha}:= \max_{k\in\X} \sup_{v\in(0,1]\cap\T} \E_k\big[\|Y_v\|^{\alpha}\big] < \infty.  
\end{equation}
The next theorem provides a CLT for $t^{-1/2}Y_t$, proved for $d:=1$ in \cite{KeiWis64} for $\T=\N$ and in \cite{FukHit67} for $\T=[0,\infty)$ (see  \cite{FerHerLed12} for $\rho$-mixing driving Markov processes). 
\begin{theo} \label{lem-clt}
Under Assumptions~\emph{\IP}~and $\M{2}$, $\{t^{-1/2}Y_t\}_{t\in \T}$ converges in distribution to a $d$-dimensional Gaussian law $\cN(0,\Sigma)$ when $t\r+\infty$. 
\end{theo} 

Now let us specify the notations and assumptions involved in our LLT. First we assume that the MAP $\{(X_t,Y_t)\}_{t\in \T}$ satisfies the following usual non-lattice condition:

\noindent \emph{\NL} : 
{\it  there is no $a\in\R^d$, no closed subgroup $H$ in $\R^d$, $H\neq \R^d$, and finally no function $\beta\, :\, \X\r\R^d$ such that: }
$$\forall k\in\X,\quad Y_1 + \beta(X_1) - \beta(k)\ \in\ a + H\ \ \ \P_{k}\text{-a.s.},$$
Second we introduce the assumptions on the density process of $t^{-1/2}Y_t$. For any $t\in\T$ and $(k,\ell)\in\X^2$, we define the  bounded positive measure $\cY_{k,\ell,t}$ on $\R^d$: 
\begin{equation} \label{loi-Yt-kl}
\forall B\in B(\R^d),\quad \cY_{k,\ell,t}(1_B) := \P_k\big\{X_t=\ell, Y_t\in B\big\}.  
\end{equation}
Let $\ell_d$ denote the Lebesgue measure on $\R^d$. From the Lebesgue decomposition of $\cY_{k,\ell,t}$ w.r.t.~$\ell_d$, there are two bounded positive measures $\cG_{k,\ell,t}$ and $\mu_{k,\ell,t}$ on $\R^d$ such that
\begin{subequations} 
\begin{eqnarray}
\forall B\in B(\R^d) &&  \cY_{k,\ell,t}(1_B):= \cG_{k,\ell,t}(1_B) + \mu_{k,\ell,t}(1_B) \label{Leb-dec-kl} \\
\text{where }  &&  \cG_{k,\ell,t}(1_B)=\int_B g_{k,\ell,t}(y)\, dy   \label{Leb-ac-kl}
\end{eqnarray}
\end{subequations}
for some measurable function $g_{k,\ell,t} : \R^d\r[0,+\infty)$, and such that $\mu_{k,\ell,t}$ and $\ell_d$ are mutually singular. The  measure $\cG_{k,\ell,t}$ is called the absolutely continuous (a.c.) part of  $\cY_{k,\ell,t}$ with associated density $g_{k,\ell,t}$. 
For any $t\in\T$, we introduce the following $N\times N$-matrices with entries in the set of bounded positive measures on $\R^d$ 
\begin{subequations}
\begin{equation} \label{def-meas-mat} 
 \cY_{t} := (\cY_{k,\ell,t})_{(k,\ell)\in\X^2},\quad  \cG_t:= (\cG_{k,\ell,t})_{(k,\ell)\in\X^2}, \quad  \cM_t := (\mu_{k,\ell,t})_{(k,\ell)\in\X^2},
\end{equation}
and for every $y\in\R^d$, the following real $N\times N$-matrix: 
\begin{equation} \label{def-Gt-density}
G_t(y) := \big(g_{k,\ell,t}(y)\big)_{(k,\ell)\in\X^2}.
\end{equation}
Then the component-wise equalities (\ref{Leb-dec-kl})-(\ref{Leb-ac-kl}) read as follows in a matrix form: for any $t\in\T$  
\begin{equation} \label{dec-leb-mat}
\forall B\in B(\R^d), \quad \cY_{t}(1_B) = \cG_{t}(1_B) + \cM_{t}(1_B) =\int_B G_t(y) dy + \cM_t(1_B). 
\end{equation}
\end{subequations}

The assumptions on the a.c.~part $\cG_t$ and the singular part $\cM_t$ of $\cY_t$ are the following ones.
\begin{ass}: \label{AS2}
There exist $c>0$ and $\rho\in(0,1)$ such that   
\begin{equation} \label{masse-id}
\forall t >0,\quad  \| \cM_{t}(1_{\R^d}) \|_{0}  \leq c\rho^t 
\end{equation}
and there exists $t_0>0$ such that 
\begin{subequations}
\begin{equation} \label{cond-c-rho}
\rho^{t_0}\max(2,cN)\leq 1/4 
\end{equation}
\begin{equation} \label{fourier-id}
\Gamma_{t_0}(\zeta) :=  \sup_{w\in[t_0,2t_0)} \| \widehat G_{w}(\zeta)\|_0 \longrightarrow 0\quad \text{when}\ \|\zeta\|\r+\infty.
\end{equation}
\end{subequations}
\end{ass}
 \begin{ass}: \label{AS0} For any $t>0$, there exists an open convex subset $\cD_t$ of $\R^d$ such that $G_t$ vanishes on $\R^d\setminus\overline{\cD}_t$, where $\overline{\cD}_t$ denotes the adherence of $\cD_t$. Moreover $G_t$ is continuous on $\overline{\cD}_t$ and differentiable on $\cD_t$, with in addition 
\begin{subequations}
\begin{gather} 
\sup_{t>0} \sup_{y\in \overline{{\cal D}}_t} \| G_t(y)\|_0 <\infty  \label{Bornes_G}\\ 
\sup_{y\in \partial {\cal D}_t}  \|G_t(y)\|_0 = O\big(t^{-(d+1)/2}\big) \quad \text{where}\ \ \partial {\cal D}_t := \overline{{\cal D}}_t\setminus \cD_t \label{Bornes_G-front} \\
 j=1,\ldots,d:\quad \sup_{t>0}\sup_{y\in {\cal D}_t} \big\| \frac{\partial G_t}{\partial y_j}(y)\big\|_0 <\infty. \label{Bornes_G-dif} 
\end{gather}
\end{subequations}
\end{ass}
The next theorem gives a  LLT for the density of the a.c.~part of the probability distribution of $t^{-1/2}Y_t$. This is the main contribution of \cite{HerLed13}.
\begin{theo} \label{main-id}
Let $\{(X_t,Y_t)\}_{t\in \T}$ be a centered MAP 
satisfying Assumptions~\emph{\IP}, $\M{3}$, \emph{\textbf{(AC\ref{AS2})-(AC\ref{AS0})}}. Moreover assume that the matrix $\Sigma$ of Theorem~\ref{lem-clt} is invertible. Then, 
for every $k\in\X$, the density $f_{k,t}(\cdot)$ of the a.c.~part   of the probability distribution of $t^{-1/2}Y_t$ under $\P_k$ satisfies the following property: 
\begin{equation} \label{result_final}
\sup_{y \in \R^d} \big|f_{k,t}(y) - \eta_{_\Sigma}(y)\big| \leq O\big( t^{-1/2}\big) + O\big(\sup_
{y\notin {\cal D}_t}\eta_{_\Sigma}(t^{-1/2}y) \big)
\end{equation}
where $\eta_{_\Sigma}(\cdot)$ denotes the density of the non-degenerate $d$-dimensional Gaussian distribution  $\cN(0,\Sigma)$ involved in Theorem~\emph{\ref{lem-clt}}. 
\end{theo}
%

\section{Semigroup of Fourier matrices and basic lemmas} \label{proof-LLT}
We assume that the conditions of Theorem~\ref{main-id} hold, and for the sake of simplicity that $\Sigma$ is the identity matrix. The proof of Theorem~\ref{main-id} involves Fourier analysis as in the i.d.d.~case. In our Markov context, this study is based on the semi-group property of the matrix family $\{\widehat\cY_t(\zeta)\}_{t\in\T}$ for every $\zeta\in\R^d$ which allows to analyze the characteristic function of $Y_t$. In this section,  we provide a collection of lemmas which highlights the connections between the assumptions of Section~\ref{sec-ass-stat} and the behavior of this semi-group.
This is the basic material for the derivation of Theorem~\ref{main-id} in \cite[Section~3]{HerLed13}.

Recall that, for every $(k,\ell)\in\X^2$, the measures $\cG_{k,\ell,t}$ (with density $g_{k,\ell,t}$) and $\mu_{k,\ell,t}$ denote the a.c.~and singular parts of the bounded positive measure $\cY_{k,\ell,t}(1_{\cdot})=\P_k\{X_t=\ell,Y_t\in \cdot\}$ on $\R^d$. 
Using the notation of (\ref{Leb-dec-kl})-(\ref{Leb-ac-kl}), the bounded positive measure $\sum_{\ell=1}^N \cG_{k,\ell,t}$ with density 
$$g_{k,t}:=\sum_{\ell=1}^N g_{k,\ell,t}$$
is the a.c.~part of the probability distribution of $Y_t$ under $\P_k$, while $\mu_{k,t}:=\sum_{\ell=1}^N\mu_{k,\ell,t}$ is its singular part. 
That is, we have for any $k\in\X$ and $t>0$: 
\begin{equation} \label{loi-Yt}
\forall B\in B(\R^d),\quad \P_k\{Y_t\in B\} = \int_B g_{k,t}(y)\, dy + \mu_{k,t}(1_B).
\end{equation}

 The bounded positive measure  $\cY_t$ is defined in (\ref{def-meas-mat}) and its Fourier transform $\widehat\cY_t$ is (see (\ref{def-Four-mat})): 
\begin{equation} \label{Fourier}
\forall (t,\zeta)\in\T\times\R^d,\ \forall (k,\ell)\in \X^2,\quad \big(\widehat\cY_t(\zeta)\big)_{k,\ell} = \E_{k}\big[1_{\{X_t=\ell\}}\, e^{i \langle \zeta , Y_t \rangle}\big]. 
\end{equation}
Note that $\widehat\cY_t(0) = P_t$. From the additivity of the second component $Y_t$, we know that that $\{\widehat\cY_t(\zeta)\}_{t\in\T,\zeta\in\R^d}$ is a semi-group of matrices (e.g. see \cite{FerHerLed12} for details), that is using the usual product in the space $\cM_N(\C)$ of complex $N\times N$-matrices:
\begin{equation} \label{semi-group}
\forall \zeta\in\R^d,\ \forall(s,t)\in\T^2,\quad \widehat\cY_{t+s}(\zeta) = \widehat\cY_t(\zeta)\, \widehat\cY_s(\zeta). \tag{SG}
\end{equation}
In particular the following property holds true
\begin{equation} \label{semi-gpe-discret}
\forall \zeta\in\R^d,\ \forall n\in\N,\qquad \widehat\cY_n(\zeta) := \big(\E_{k}[1_{\{X_n=\ell\}}\, e^{i \langle \zeta , Y_n \rangle}]\big)_{(k,\ell)\in\X^2} = \widehat\cY_1(\zeta)^n. 
\end{equation}
For every $k\in\X$ and $t\in\T$, we denote by $\phi_{k,t}$ the characteristic function of $Y_t$ under $\P_k$: 
\begin{equation} \label{car-fct}
\forall \zeta\in\R^d,\quad \phi_{k,t}(\zeta) := \E_{k}\big[e^{i \langle \zeta , Y_t \rangle}\, \big] = e_k \widehat\cY_t(\zeta) \mathbf{1}^{\top},
\end{equation}
where $e_k$ is the $k$-th vector of the canonical basis of $\R^d$. 

The first lemma below provides the control of $\phi_{k,t}$ on a ball $B(0,\delta):= \{\zeta\in\R^d : \|\zeta\| < \delta\}$ for some $\delta>0$, using a perturbation approach. The second one is on the control of $\phi_{k,t}$ on the annulus $\{\zeta\in\R^d : \delta\leq\|\zeta\|\leq A]$ for any $A>\delta$. Finally the third lemma focuses on the Fourier transform of the density $g_{k,\ell,t}$ of the a.c. part $\cG_{k,\ell,t}$ of $\cY_{k,\ell,t}$ in (\ref{Leb-dec-kl}) on the domain $\{\zeta\in\R^d :\|\zeta\|\geq A\}$ for some $A>0$.  
\begin{lem}[\protect{\cite[Lem.~4.1]{HerLed13}}] \label{lem-dec-vp}
Under Assumptions~\emph{\IP}~and $\M{3}$, there exists a real number $\delta>0$ 
such that, for all $\zeta\in B(0,\delta)$, the characteristic function of $Y_t$ satisfies 
\begin{equation} \label{exp-val-pert}
\forall k\in\X,\ \forall t\in\T,\quad \phi_{k,t}(\zeta)  = \lambda(\zeta)^{\lfloor t \rfloor}\, L_{k,t}(\zeta)  + R_{k,t}(\zeta),
\end{equation}
with $\C$-valued functions $\lambda(\cdot)$, $L_{k,t}(\cdot)$ and $R_{k,t}(\cdot)$ on $B(0,\delta)$ satisfying the next properties for $t\in\T$, $\zeta\in\R^d$ and $k\in\X$: 
\begin{subequations}
\begin{eqnarray}
& &  \|\zeta\| < \delta\ \Rightarrow\ \lambda(\zeta) = 1 - \frac{\|\zeta\|^2}{2} + O(\|\zeta\|^3) \label{exp-val-pert0} \\
& & t\geq 2,\ \|t^{-1/2}\zeta\| < \delta\ \Rightarrow\  \big|\lambda(t^{-1/2}\zeta)^{\lfloor t \rfloor} - e^{-\|\zeta\|^2/2}\big| \leq C\, \frac{(1+\|\zeta\|^3)e^{-\|\zeta\|^2/8}}{\sqrt t} \label{exp-val-pert0-bis} \\
& & \|\zeta\| < \delta\ \Rightarrow\ \big|L_{k,t}(\zeta)-1\big| \leq C\, \|\zeta\|  \label{exp-val-pert1} \\
& & \exists\, r\in(0,1),\quad \sup_{\|\zeta\|\leq \delta}\, |R_{k,t}(\zeta)| \leq C\, r^{\lfloor t \rfloor}  \label{exp-val-pert2}
\end{eqnarray}
\end{subequations}
where the constant $C>0$ in \emph{(\ref{exp-val-pert0-bis})}-\emph{(\ref{exp-val-pert2})} only depends on $\delta$ and on $M_3$ in $\M{3}$ (see \emph{(\ref{moment-alpha})}). 
\end{lem}
\begin{proof}{ of \emph{(\ref{exp-val-pert0-bis})}}
From (\ref{exp-val-pert0}) we obtain for $v\in\R^d$ such that $|v|\leq \delta$ (up to reduce $\delta$)
$$|\lambda(v)| \leq 1 - \frac{\|v\|^2}{2} + \frac{\|v\|^2}{4}  \leq e^{- \|v\|^2/4}.$$
Therefore we have for any $t\geq 2$ and any $\zeta\in\R^d$ such that $t^{-1/2}\|\zeta\|\leq \delta$,
\begin{equation} \label{Z}
\big|\lambda\big(t^{-1/2}\zeta\big)\big| 
\leq e^{-\frac{\|\zeta\|^2}{4t}}. 
\end{equation}
Now consider any $t\geq 2$ and any $\zeta\in\R^d$ such that $t^{-1/2}\|\zeta\|\leq \delta$. Set $n:=\lfloor t \rfloor$, and write
$$\lambda\big(t^{-1/2}\zeta\big)^n - e^{-\frac{\|\zeta\|^2}{2}} =
\big(\lambda\big(t^{-1/2}\zeta\big) -
e^{-\frac{\|\zeta\|^2}{2n}}\big) \, \sum_{k=0}^{n-1}
\lambda\big(t^{-1/2}\zeta\big)^{n-k-1}
e^{\frac{-k\|\zeta\|^2}{2n}}.$$
We have 
\begin{eqnarray*}
\big|\sum_{k=0}^{n-1}
\lambda\big(t^{-1/2}\zeta\big)^{n-k-1}
e^{\frac{-k\|\zeta\|^2}{2n}}\big| &\leq& \sum_{k=0}^{n-1}
\big|\lambda\big(t^{-1/2}\zeta\big)\big|^{n-k-1} 
e^{\frac{-k\|\zeta\|^2}{2n}} \\
&\leq& \sum_{k=0}^{n-1}
e^{-\frac{\|\zeta\|^2(n-k-1)}{4t}} 
e^{\frac{-k\|\zeta\|^2}{4t}} \\
&\leq& n\, e^{-\frac{n\|\zeta\|^2}{4t}} 
e^{\frac{\|\zeta\|^2}{4t}} \leq  b\, t\, e^{-\frac{\|\zeta\|^2}{8}} 
\end{eqnarray*}
where $b := \sup_{|u|\leq \delta}\exp(u^2/4)$, since $n/t \geq (t-1)/t \geq 1/2$ for $t\geq2$. Moreover, using (\ref{exp-val-pert0}) and the fact that $|e^a-e^b| \leq |a-b|$ for any $a,b\in(-\infty,0)$, we obtain that there exist positive constants $D$ and $D'$ such that 
\begin{eqnarray*}
\big|\lambda\big(t^{-1/2}\zeta\big) -
e^{-\frac{\|\zeta\|^2}{2n}}\big| &\leq& \big|\lambda\big(t^{-1/2}\zeta\big) -
e^{-\frac{\|\zeta\|^2}{2t}}\big| + \big|e^{-\frac{\|\zeta\|^2}{2t}} -
e^{-\frac{\|\zeta\|^2}{2n}}\big| \\ 
&\leq& D\, t^{-\frac{3}{2}}\|\zeta\|^3 + 
\frac{\|\zeta\|^2}{2}\big|\frac{1}{t} - \frac{1}{n}\big| \\
&\leq& D\, t^{-\frac{3}{2}}\|\zeta\|^3 + 
\frac{\|\zeta\|^2}{2}\, \frac{1}{t(t-1)} 
\\
&\leq& D'\big(t^{-\frac{3}{2}}\|\zeta\|^3 + t^{-2}\|\zeta\|^2\big).  
\end{eqnarray*}
Thus 
\begin{eqnarray*}
\bigg|\lambda\big(t^{-1/2}\zeta\big)^n - e^{-\frac{\|\zeta\|^2}{2}}\bigg| &\leq& 
bD'\, t\, e^{-\frac{\|\zeta\|^2}{8}} \big(t^{-\frac{3}{2}}\|\zeta\|^3 + 
 t^{-2} \|\zeta\|^2\big) \\ 
&\leq& bD'\, e^{-\frac{\|\zeta\|^2}{8}} \bigg(\frac{\|\zeta\|^3}{\sqrt t} + 
 \frac{\|\zeta\|^2}{t}\bigg) \\
 &\leq& \frac{bD'}{\sqrt t}\, e^{-\frac{\|\zeta\|^2}{8}} \big(\|\zeta\|^3 + 
 \|\zeta\|^2\big)\\
 &\leq& \frac{D''}{\sqrt t}\,\, e^{-\frac{\|\zeta\|^2}{8}} \big(1+\|\zeta\|^3\big)
\end{eqnarray*}
for some positive constant $D''$. 
\end{proof}
\begin{lem}[\protect{\cite[Lem.~4.2]{HerLed13}}] \label{lem-non-ari} 
Assume that Conditions~\emph{\IP}, \emph{\textbf{(AC1)}} and $\M{\alpha}$ for some $\alpha>0$ hold. Let $\delta,A$ be any real numbers such that $0<\delta<A$. There exist constants $D\equiv D(\delta,A)>0$ and $\tau\equiv \tau(\delta,A)\in(0,1)$ such that 
\begin{equation} \label{ineg-non-ari}
\forall k\in\X,\ \forall t\in\T,\quad \sup_{\delta\leq\|\zeta\|\leq A} |\phi_{k,t}(\zeta)| \leq D\, \tau^{\lfloor t \rfloor}.
\end{equation}
\end{lem}
\begin{lem}[\protect{\cite[Th.~4.3]{HerLed13}}] \label{lem-hat-g}
Under Condition~\emph{\textbf{(AC\ref{AS2})}}, there exist positive constants $A$ and $C$ such that the following property holds:  
$$|\zeta|\geq A\ \Longrightarrow\ \forall (k,\ell)\in\X^2,\ \forall t\in [t_0,+\infty[, 
\quad \big|\widehat g_{k,\ell,t}(\zeta)\big| \leq C\, \frac{t}{2^{t/t_0}}.$$
\end{lem}
%

\section{Application to the Markov Renewal Processes} \label{MRP}

Let $\{X_n,Y_n\}_{n\in\N}$ be a discrete-time  MAP with state space $\X\times \R $ and $\X:=\{1,\ldots,N\}$. When $\{\xi_n\}_{n\ge 0}$, with $\xi_0:=Y_0$ and   $\xi_n:=Y_n-Y_{n-1}$ for $n\ge 1$,   is a sequence of non-negative random variables, then $\{X_n,Y_n\}_{n\in\N}$ is also known as a Markov Renewal Process (MRP). The so-called semi-Markov kernel $Q(\cdot;\{\cdot\}\times dy)$  is defined by  (e.g. see \cite[VII.4]{Asm03})
	\begin{equation}\label{semi_Markov} \forall B \in B(\R), \ \forall (k,\ell)\in \X^2, \quad  \P_k\{X_{1}=\ell, Y_{1} \in B\} = \int_B  Q(k;\{\ell\}\times dy)
\end{equation}
and is nothing else but the measure $\cY_{k,\ell,1}$ in (\ref{loi-Yt-kl}). 
The transition probability matrix $P$ associated with the Markov chain $\{X_n\}_{n\in\N}$ is given by 
	\[\forall (k,\ell)\in \X^2, \quad P(k,\ell) = Q(k;\{\ell\}\times \R)=\int_{\R} Q(k;\{\ell\}\times dy).
\]
  For any $n\ge 2$, the bounded positive measure $\cY_{k,\ell,n}$ is defined by the convolution product of the semi-Markov kernel $Q$, that is 
	\begin{equation} \label{71}
\forall B \in B(\R^d), \quad	 \cY_{k,\ell,n}(1_B):= \P_k\{ X_n=\ell, Y_n\in B\} = Q^{\star n}(k,\{\ell\}\times B)
\end{equation}
Then, the Theorem~\ref{main-id} for the density process of $Y_n/\sqrt{n}$ could be specified to this specific class of MAPs. Since we only have to replace time $t$ by $n$ in all the material developed in Section~\ref{sec-ass-stat}, we omit the details. Note that the only simplification in assumptions of  Theorem~\ref{main-id} is on the uniform moment condition $\M{\alpha}$ which reduces to: the r.v.~$Y_1$ satisfies the following moment condition of order $\alpha$: 
\begin{equation} \label{moment-alpha-discret} 
M_{\alpha}:= \max_{k\in\X}  \E_k\big[\|Y_1\|^{\alpha}\big] < \infty.  
\end{equation} 
We will only illustrate our main result on the MRP embedded in a Markovian Arrival Process (e.g. see \cite[XI.1]{Asm03}).

 Recall that a $N$-state Markovian Arrival Process is a continuous-time MAP $\{(J_t,N_t)\}_{t \ge 0}$ on the state space $\{1,\ldots ,N\}\times \N$,  where $N_t$ represents the number of arrivals up to time $t$, while the states of the driving Markov process
$\{J_t\}_{t \ge 0}$ are called phases. Let $Y_n$ be the time at the $n$th arrival ($Y_0=0$ a.s.) and let $X_n$ be the state of the driving process just after the $n$th arrival. Then $\{(X_n,Y_n)\}_{n\in\N}$ is known to be an MRP with the following semi-Markov kernel $Q$ on $\{1,\ldots,N\}\times \R$: for any $(k,\ell)\in\X^2$  
	\begin{equation} \label{70}
	  Q(k;\{\ell\}\times dy) := 
	(e^{yD_0 }D_1)(k,\ell) \, 1_{(0,\infty)}(y) dy = e_k e^{yD_0}D_1 {e_{\ell}}^{\top}\, 1_{(0,\infty)}(y) \, dy
\end{equation}
which is parametrized by a pair of $N\times N$-matrices usually denoted by $D_0$ and $D_1$. Such a process is also an instance of  Markov Random Walk. 
The matrix $D_0+D_1$ is the  infinitesimal generator of the background Markov process $\{J_t\}_{t \ge 0}$. Matrix $D_0$ is always assumed to be stable and $D_0+D_1$ to be irreducible. In this case, $\{J_t\}_{t \ge 0}$ has a unique invariant probability measure denoted by $\pi$. 
The process $\{X_n\}_{n\in\N}$ is a Markov chain with state space $\X:=\{1,\ldots,N\}$ and transition probability matrix $P$:
\begin{equation} \label{P_MRP}
 \forall (k,\ell)\in\X^2, \quad P(k,\ell) =  Q(k;\{\ell\}\times \R) = \left((-D_0)^{-1}D_1\right)(k,\ell).
\end{equation}
This Markov chain has an invariant probability measure $\boldsymbol{\phi}$ (different of $\pi$). From (\ref{70}), the bounded positive measure $\cY_{k,\ell,1}$ is  
 absolutely continuous with respect to the Lebesgue measure on $\R$ with density $g_{k,\ell,1}(y) =  e_k e^{yD_0} D_1 {e_{\ell}}^{\top} 1_{(0,\infty)}(y)$. Then matrix $G_t$ defined in (\ref{def-Gt-density}) has the form 
	\[ G_1(y) = e^{yD_0} D_1 \, 1_{(0,\infty)}(y).
\]
For any $n\ge 2$, the positive measure $\cY_{k,\ell,n}$ in (\ref{71}) is absolutely continuous with respect to the Lebesgue measure with density given by 
\begin{equation} \label{convol}
\forall y\in\R, \quad  G_n(y) = (G_{n-1} \star G_1 )(y) =: \bigg( \sum_{j\in\X} \big( G_{k,j,1} \star  G_{j,\ell,n-1}\big)(y)\bigg)_{k,\ell\in\X^2}.
\end{equation}

Let us check the assumptions of Theorem~\ref{main-id}. Assumption~{\IP}~on $P$ in (\ref{P_MRP}) is standard in the literature on the Markov Arrival Process.   It is well known that  $Y_1$ has a moment of order $3$ given by $E_k[(Y_1)^3] = 3!\, e_k (-D_0)^{-3}\,\mathbf{1}^{\top}$, so that  Condition~(\ref{moment-alpha-discret}) holds with $\alpha=3$. Let us give some details for checking  Conditions~\textbf{(AC\ref{AS2})-(AC\ref{AS0})}.

\noindent\textbf{(AC\ref{AS0})} : The open convex $\cD_n$ involved in \textbf{(AC\ref{AS0})} is given for any $n\ge 1$  by 
	\[ \cD_n:=]0,\infty[, \quad \text{ with } \overline{\cD}_n=[0,\infty), \quad \partial \cD_n = \{0\}.
\]
It is clear that $G_1(y)$  is continuous on $\overline{\cD}_1$ and differentiable on $\cD_1$ with differential 
\begin{equation}\label{G1-deriv}
 \forall y >0,\quad \frac{dG_{1}}{dy}(y) = D_0 e^{yD_0} D_1. 
\end{equation}
Since $G_n$ is obtained from convolution product of $G_1$, the same properties are also valid for $G_n$. 
Next, we prove by induction that 
\begin{equation} \label{Gn-borne}
	 \sup_{n>0}\sup_{y\in[0,+\infty)}\|G_n(y)\|_0 \le N\|D_1\|_0.
\end{equation}
For $n:=1$, we have from norm equivalence (\ref{equiv-norm}) and $\| e^{D_0}\|_{\infty}\le 1$ that 
\begin{eqnarray*}
 \forall y\in [0,+\infty),\quad \|G_1(y)\|_0 		& = & \| e^{yD_0} D_1\|_0 \le \|e^{yD_0} D_1\|_{\infty} \le \| e^{yD_0} \|_{\infty} \|D_1\|_{\infty} \\
	& \le & N \| D_1 \|_{0}. 
\end{eqnarray*}
Using the definition (\ref{convol}) of $G_n(y)$ and the induction hypothesis, we have 
\begin{eqnarray*}
 | G_{k,\ell,n}(y)| & = & \sum_j \int_{\R} |G_{k,j,1}(u)| |G_{j,\ell,n-1}(t-u)| du \le N \|D_1\|_0 \int_{\R} G_{k,j,1}(u) du  \\
 & \le &  N \|D_1\|_0 \, \P_k\{X_1=\ell\} \le  N \|D_1\|_0.
\end{eqnarray*}
The proof of estimate (\ref{Gn-borne}) is complete.

Second, we easily obtain from (\ref{G1-deriv}) that 
$$ \sup_{y \in(0,+\infty)}\|\frac{dG_1}{dy}(y)\|_0 \le \big(N \max(\|D_0\|_0,\|D_1\|_0)\big)^2  $$
Next, using an induction, we can obtain from the following well known property of the convolution product  
$$\frac{dG_n}{dy}(y) = \bigg( G_{1} \star \frac{dG_{n-1}}{dy}\bigg)(y)$$
that 
\[ \sup_{n>0} \sup_{y\in (0,+\infty)} \| \frac{dG_n}{dy}(y)\|_0 \le  \big(N\max(\|D_0\|_0, \|D_1\|_0)\big)^2. 
\]
Finally, $G_1(0) = D_1$ and  $G_n(0)= 0$ for any $n\ge 2$.

\noindent\textbf{(AC\ref{AS2})} :  First, observe that $\mu_{k,\ell,n}$ is 0 for any $(k,\ell)\in\X^2$ and $n\ge 1$. Thus (\ref{masse-id}) is satisfied for $\rho=0$. In Condition~(AC\ref{AS2}), we only have to check that 
	\[ \| \widehat G_{n_0}(\zeta)\|_0 \longrightarrow 0\quad \text{when}\ |\zeta|\r+\infty
\]
for $n_0$ large enough. It follows from the convolution definition (\ref{convol}) of $G_n$ that  
\begin{equation*} 
\forall n\in\N,\ \forall \zeta\in\R,\qquad \widehat G_n(\zeta) := \big(\E_{k}\big[1_{\{X_n=\ell\}}\, e^{i \langle \zeta , Y_n \rangle}\big]\big)_{(k,\ell)\in\X^2} = \widehat G_1(\zeta)^n
\end{equation*}
where 
\begin{equation*} \label{Fourier_MAP}
	\widehat G_1(\zeta) =  \int_0^{+\infty} e^{i \zeta  y} e^{D_0 y}dy \, D_1 
\end{equation*}
 is integrable. In fact, using an integration by parts, we obtain that 
\begin{eqnarray*}
\forall \zeta \neq 0, \qquad 	\widehat G_1(\zeta) & = & \frac{-1}{i\zeta} \big[ D_1 + \int_0^{+\infty} e^{i\zeta y }D_0 e^{yD_0} D_1\, dy\big]
\end{eqnarray*}
so that 
	\[ \forall \zeta\neq 0, \qquad  \|\widehat G_1(\zeta)\|_0 \le \frac{2\| D_1\|_0}{|\zeta|}.
\]
Therefore,  we deduce from norm equivalence (\ref{equiv-norm}) that for any integer  $n_0\ge 1$,
	\[  \| \widehat G_{n_0}(\zeta)\|_0 
	\le \frac{(2N\|D_1\|_{0})^{n_0}}{|\zeta|^{n_0}} \longrightarrow 0\quad \text{when}\ |\zeta|\r+\infty.
\]

\section{Application to the local times of a jump process} \label{appli-local-times}

\subsection{The local limit theorem for the density of local times}

Let $\{X_t\}_{t\ge 0}$ be a Markov jump process  with finite state space $\X:=\{1,\ldots,N\}$ and generator $G$. Its transition semi-group is given by 
	\[ \forall t\ge 0, \quad P_t := e^{tG}.
\]
The local time $L_t(i)$ at time $t$ associated with state $i\in\X$, or the sojourn time in state $i$ on the interval $[0,t]$,  is defined by 
	\[ \forall t\ge 0, \quad L_t(i):= \int_0^t 1_{\{X_s=i\}} \, ds.
\]
It is well known that $L_t(i)$ is an additive functional of $\{X_t\}_{t\ge 0}$ and that $\{(X_t,L_t(i))\}_{t\ge 0}$ is an MAP. 
In this section, we consider the MAP $\{(X_t,L_t)\}_{t\ge 0}$ where $L_t$ is the random vector of the local times 
$$L_t:= (L_t(1),\ldots,L_t(N)).$$
 Note that, for all $t>0$, we have $\langle L_t, \mathbf{1}\rangle =t$, that is $L_t$ is $\cS_t$-valued where 
$$\cS_t := \big\{ y\in [0,+\infty)^N: \langle y, \mathbf{1}\rangle =t\big\}
.$$ 
Let us assume that $\{X_t\}_{t\ge 0}$ has an invariant probability measure $\pi$. This happens when the generator $G$ is irreducible. Set $m= (m_1,\ldots,m_N):=\E_{\pi}[L_1]$.  
We define the $\cS_{t}^{(0)}$-valued centered r.v. $$Y_t = L_t - tm$$ where $\cS_{t}^{(0)}:=T_{-tm}(\cS_t)$ with the translation $T_{-tm}$ by vector $-tm$ in $\R^N$. Note that  $\cS_{t}^{(0)}$ is a subset of the hyperplane $H$ of $\R^N$ defined by  
\begin{equation} \label{Def_H_Local}
H:=\big\{y\in\R^N: \ \langle y, \mathbf{1}\rangle = 0\big\}.
\end{equation}
Let $\Lambda$ be the bijective map from $H$ into $\R^{N-1}$ defined by $\Lambda(y):=(y_1,\ldots,y_{N-1})$ for $y:=(y_1,\ldots,y_N)\in H$. 
We introduce the following $(N-1)$-dimensional random vector  
$$Y'_t:=\Lambda(Y_t) = (Y_t(1),\ldots,Y_t(N-1)).$$
The following lemma follows from Lemma~\ref{lem_TMAP}.
\begin{lem}
The process $\{(X_t,Y'_t)\}_{t\ge 0}$ is a $\X\times\overline{\cD}_t$-valued MAP where $\cD_t$ is the open convex of $\R^{N-1}$ defined by 
\begin{equation} \label{Def_Dt}
	\cD_t := \big\{ y'\in \R^{N-1} : j=1,\ldots, N-1, \ y'_j \in (-m_j t, (1-m_j)t), \ \langle y',\mathbf{1}\rangle < m_N t \big\}.
\end{equation} 
\end{lem}
\begin{pro} \label{Lem_Jointe}
If $G$ and the sub-generators $G_{i^ci^c}:=(G(k,\ell))_{k,\ell\in\{i\}^c}, \,i=1,\ldots,N$ are irreducible then the MAP $\{(X_t,Y'_t)\}_{t\ge 0}$ satisfies the conditions $\M{3}$ and \emph{\textbf{(AC1)-(AC2)}}. 
\end{pro}
Note that the matrix $\Sigma$ in the CLT for $\{t^{-1/2}Y'_t\}_{t>0}$ is invertible from \cite[Remark~2.3]{HerLed13} and that $O(\sup_
{y\notin {\cal D}_t}\eta_{_\Sigma}(t^{-1/2}y)) = O(t^{-1/2})$ from (\ref{Def_Dt}). Under the assumptions of Proposition~\ref{Lem_Jointe} on the generator of $\{X_t\}_{t\ge 0}$, Theorem~\ref{main-id} gives that 
$$\sup_{y' \in \R^{N-1}} \big|f'_{k,t}(y') - (2\pi)^{-(N-1)/2}(\det \Sigma)^{-1/2} e^{-\frac{1}{2}\langle {y'}^{\top},\Sigma {y'}^{\top} \rangle}\big| = O\big(t^{-1/2}\big),$$
where $f'_{k,t}$ is the density of the a.c.~part of the probability distribution of $t^{-1/2}Y'_t$. An explicit form of $\Sigma$ is provided in \cite[Remark~3.1]{HerLed13}. 

In order to obtain a statement in terms of $Y_t$, we can use the following lemma. Recall that $\ell_{N-1}$ denotes the Lebesgue measure on $\R^{N-1}$ and $B(H)$ stands for the Borelian $\sigma$-algebra on $H$. Let $\nu$ be the measure defined on $(H,B(H))$ as the image measure of $\ell_{N-1}$ under $\Lambda^{-1}$, that is: for any  positive and $B(H)$-measurable function $\psi : H\r\R$, 
$$\int_H \psi(y)\, \nu(dy) := \int_{\R^{N-1}} \psi(\Lambda^{-1}y')\, dy'.$$
Using the bijection $\Lambda$, the following lemma can be easily proved (see Appendix~\ref{B}).
\begin{lem} \label{Lem52} For any $(k,\ell)\in\X^2$, the a.c.~and singular parts of the Lebesgue decomposition of the measure $\P_k\{X_t=\ell,Y_t\in \cdot\}$ on $H$ with respect to $\nu$  are obtained from image measure under  $\Lambda$ of the a.c.~and singular parts of the Lebesgue decomposition of the measure $\P_k\{X_t=\ell,Y'_t\in \cdot\}$ on $\R^{N-1}$ with respect to $\ell_{N-1}$  . 
\end{lem}
Therefore we know that the density, say $f_{k,t}$, of the a.c.~part of the probability distribution of 
$$t^{-1/2}Y_t=t^{-1/2}(L_t-tm)$$ with respect to the measure $\nu$ on the hyperplane $H$ is given by  
$$\forall h\in H,\quad f_{k,t}(h) := f'_{k,t}(\Lambda h).$$
Finally, we have obtained the following local limit theorem for $f_{k,t}$.
\begin{pro}[\protect{\cite[Prop.~3.1]{HerLed13}}]
If $G$ and the sub-generators $G_{i^ci^c}:=(G(k,\ell))_{k,\ell\in\{i\}^c}, \,i=1,\ldots,N$ are irreducible, then the density $f_{k,t}$ of the a.c.~part of the probability distribution of $t^{-1/2}(L_t-tm)$ under $\P_k$ satisfies 
$$\sup_{h \in H} \big|f_{k,t}(h) - (2\pi)^{-(N-1)/2}(\det \Sigma)^{-1/2} e^{-\frac{1}{2}\langle \Lambda h^{\top},\Sigma \Lambda h^{\top} \rangle}\big| = O\big(t^{-1/2}\big).$$
\end{pro}
It remains to prove that Proposition~\ref{Lem_Jointe} holds true.

\subsection{The main lines of the derivation of Proposition~\ref{Lem_Jointe}} \label{main_lines_proof}

Let us introduce the randomized Markov chain $\{Z_n\}_{n\in\N}$ with state space $\X$ and transition matrix 
$$\widetilde P:= I + G/a \quad \text{with } a > \max(|G(j,j)|, j \in\X).$$
Since $G$ is assumed to be irreducible, the transition matrix $\widetilde{P}$ is irreducible and aperiodic. Moreover, since $\widetilde P(i,i)>0$ for any $i\in\X$, the sub-stochastic matrix $\widetilde P_{i^ci^c}:=(\widetilde P(k,\ell))_{k,\ell \in\{i\}^c}$ is aperiodic.  

Let us define the following convex open set of $\R^{N-1}$
	\[ \forall t>0, \quad \cC_t := \big\{ y\in]0,t[^{N-1}, \langle y,\mathbf{1} \rangle <t\big\}.
\]
For any  $t>0$, the adherence of $\cC_t$ is denoted by $\overline{\cC}_t$ and is $\overline{\cC}_t := \big\{ y\in[0,t]^{N-1}, \langle y,\mathbf{1} \rangle \le t\big\}$
and the boundary of $\cC_t$ is $\partial \cC_t := \overline{\cC}_t \backslash \cC_t$ given by
\begin{equation} \label{frontier}
 \partial \cC_t = \bigcup_{i=1}^N \big\{ y=(y_1,\ldots,y_{N-1})\in \overline{\cC}_t \mid y_i:=0\big\} \cup \big\{ y\in \overline{\cC}_t \mid \langle y,\mathbf{1} \rangle =t\big\}. 
\end{equation} 
The joint conditional distribution of $(X_t,L'_t):=(X_t,L_t(1),\ldots,L_t(N-1))$ under $\P_k$ is 0 on $\R^{N-1} \backslash \overline{\cC}_t$ and its a.c. part has the following density density $\psi_{k,\ell,t}$ from \cite[Cor. 4.4]{Ser99} 
\begin{eqnarray} \label{densite_jointe}
\lefteqn{	\forall (k,\ell)\in\X^2, \forall y\in\R^{N-1}, \ \psi_{k,\ell,t}(y) := 1_{\mathcal{C}_t}(y)\, a^{N-1} \sum_{n=0}^{\infty} e^{-a t} \frac{(a t)^n}{n!}}\\
 &  & \times  \sum_{\tiny \begin{array}{c}k_1\ge 0,\ldots, k_{N-1}\ge 0,\\ \sum_{j=1}^{N-1} k_j\le n\end{array}}  \, x_{n;k_1,\ldots,k_{N-1}}^{t;y}\,  p_{k,\ell}(n+N,k_1,\ldots,k_{N-1}) \nonumber
\end{eqnarray}
where the non-negative coefficient $p_{k,\ell}(n+N,k_1,\ldots,k_{N-1})$ satisfies $0\le \sum_{\ell=1}^{N-1} p_{k,\ell}(n+N,k_1,\ldots,k_{N-1})\le 1$ and we have the following relation 
\begin{equation} \label{binomial}
 \sum_{\tiny \begin{array}{c}k_2\ge 0,\ldots, k_N\ge 0,\\ \sum_{j=2}^N k_j\le n\end{array}} n! \, x_{n;k_2,\ldots,k_N}^{t;y_2,\ldots,y_N}=1
\end{equation}

For every $(k,\ell)\in\X^2$, the Lebesgue decomposition of the positive measure $\cF_{k,\ell,t}$ defined by  $\P_k\{X_t=\ell,L'_t \in \cdot\}$ writes as follows 
 \begin{equation} \label{Lebesgue_Psi}
  \forall B \in B(\R^{N-1}), \quad \cF_{k,\ell,t}(1_B) = \P_k\{X_t=\ell,L'_t \in B\} = \int_B \psi_{k,\ell,t}(y)\, dy + \alpha_{k,\ell,t}(1_B).
  \end{equation}

First note that Condition \IP~is fulfilled since the matrix $P:=P_1= e^{G}$ is irreducible and aperiodic
when the generator $G$ of the driving jump process is irreducible. Second, the random variables $\|L_t\|$ (and so $\|Y_t\|$) are bounded, 
so that the moment condition $\M{\alpha}$ is satisfied for any $\alpha >0$. This allows us to apply Theorem~\ref{lem-clt} to derive a CLT for $\{t^{-1/2}Y'_t\}_{t\ge 0}$. Next, the main steps of the proof are to check  that Conditions~\textbf{(AC\ref{AS2})-(AC\ref{AS0})} hold under the assumptions of Proposition~\ref{Lem_Jointe} on the generator $G$.

\subsubsection{Checking Condition (AC\ref{AS0})}

For every $y\in\R^{N-1}$, we introduce the following real $N\times N$-matrix
	\[ \Psi_t(y):= (\psi_{k,\ell,t}(y))_{(k,\ell)\in\X^2}.
\]
where $\psi_{k,\ell,t}$ is defined in (\ref{densite_jointe}). The useful properties of  $\psi_{k,\ell,t}$ are given in the next lemma which is proved in Appendix~\ref{annexeA}.
\begin{lem} \label{Lem_Psi}
For any $t>0$, $\Psi_t$ vanishes on $\R^d \backslash \overline{\cC}_t$, is continuous on $\overline{\cC}_t$ and differentiable on $\cC_t$. We have 
\begin{subequations}
\begin{gather} \label{Bornes_jointe_Psi}
\sup_{t>0} \sup_{y\in \R^{N-1}} \|\Psi_t(y)\|_0 \le \sup_{t>0} \sup_{y\in \R^{N-1}} \|\Psi_t(y)\mathbf{1}^{\top}\|_0 \le a^{N-1}; \\
j=1,\ldots,N-1 \qquad \sup_{t>0} \sup_{y\in\mathcal{C}_t} \big\| \frac{\partial \Psi_t}{\partial y_j}(y)\big\|_0 \le 2a^{N}.\label{Borne_deriv_Psi}
\end{gather}
There exists $\rho\in(0,1)$ such that 
\begin{equation} \label{Borne_bord_Psi}
	\forall y\in\partial \cC_t, \quad \|\Psi_t(y)\|_0 = O(e^{at(\rho-1)}(1+at)). 
\end{equation}
\end{subequations}
\end{lem}

Now, let us check that Conditions~\textbf{(AC2)} hold for the function 
\begin{eqnarray}
	G_t(y) = \Psi_t(y+m't)
\end{eqnarray}
associated with the absolutely continuous part of the probability distribution of the centered r.v. $(X_t,Y'_t) = (X_t,L'_t - m't)$ under $\P_k$ and $m':=(\pi_1,\ldots,\pi_{N-1})$. It vanishes on the open convex $\cD_t := T_{-m't}(\cC_t)$ which is the translate of $\cC_t$ by the translation $T_{-m't}$ of vector $-m't$. More precisely, 
\begin{subequations}
\begin{equation*} 
	\cD_t = \big\{ y\in \R^{N-1} : j=1,\ldots, N-1, \ y_j \in (-m_j t, (1-m_j)t), \, \langle y,\mathbf{1} \rangle < m_{N}t \big\}.
\end{equation*}
with adherence 
\begin{equation} \label{Def_barDt}
	\overline{\cD}_t = \big\{ y\in \R^{N-1} : j=1,\ldots, N-1, \ y_j \in [-m_j t, (1-m_j)t], \, \langle y,\mathbf{1} \rangle \le m_N t \big\}.
\end{equation}
and boundary 
\begin{equation} \label{Def_frontier_Dt}
\partial \cD_t = \bigcup_{j=1}^{N-1}\big\{ y\in \overline{\cD}_t \mid  y_j=-m_j t \bigg\} 
	\cup \big\{y\in \overline{\cD}_t :  \langle y,\mathbf{1} \rangle = m_N t \big\}.
\end{equation}
\end{subequations}
Then it follows from the basic properties of function $\Psi_t$ on $\cC_t$ stated in Lemma~\ref{Lem_Psi} that $G_t$ is continuous on $\overline{\cD}_t$ and differentiable on $\cD_t$ with differential 
	\[ \forall j=1,\ldots, N-1,  \ \forall y\in \cD_t, \quad \frac{\partial G_t}{\partial y_j}(y) = \frac{\partial \Psi_t}{\partial y_j}(y+m't) .
\]
Moreover, we obtain from (\ref{Bornes_jointe_Psi})-(\ref{Borne_bord_Psi})
\begin{gather*}
	 \sup_{t>0} \sup_{y\in \R^{N-1}} \|G_t(y)\|_0 \le 
	 \sup_{t>0} \sup_{y\in \R^{N-1}} \|G_t(y)\mathbf{1}^{\top}\|_0 \le \sup_{t>0} \sup_{y\in \R^{N-1}} \|\Psi_t(y)\mathbf{1}^{\top}\|_0 \le a^{N-1}; \\
	 		\forall y\in\partial \mathcal{D}_t, \quad \|G_t(y)\|_0 = \| \Psi_t(y+m't)\|_0 = O\big(\frac{1}{t}\big);\\
	\sup_{t>0} \sup_{y\in\mathcal{D}_t} \big\|  \frac{\partial G_t}{\partial y}(y)\big\|_0 \le \sup_{t>0} \sup_{y\in\mathcal{C}_t} \big\| \frac{\partial \Psi_t}{\partial y}(y)\big\|_0 \le 2a^{N}.
\end{gather*}

\subsubsection{Checking Condition (AC1)}

For any $i\in\{1,\ldots,N\}$, we introduce the following $(N-1)\times (N-1)$-subgenerator of $G$, $G_{i^ci^c} := (G(k,\ell))_{k,\ell\in \{i\}^c}$. 
If $G_{i^ci^c}$ is irreducible then $\|e^{t G_{i^ci^c}}\|_0=O(e^{-r_i t})$ where $-r_i$ is the Perron-Frobenius negative eigenvalue of $G_{i^ci^c}$. Thus, if $G_{i^ci^c}$ is irreducible for any $i\in\{1,\ldots,N\}$
\begin{equation} \label{Perron-Frobenius}
 \max_{i\in\{1,\ldots,N\}} \|e^{tG_{i^ci^c}}\|_0 = O(e^{-rt}) \qquad \text{where } r:=\min_i(r_i)>0.
\end{equation}

In a first step, we study the singular part $\cA_t$ of the Lebesgue decomposition (\ref{Lebesgue_Psi}) 
where $\cA_t$ is the $N\times N$-matrix with entries in the set of bounded positive measures on $\R^{N-1}$
 $$\cA_t :=(\alpha_{k,\ell,t})_{(k,\ell)\in\X^2}.$$ 
We show that $\|\cA_t(1_{\R^{N-1}})\|_0$ goes to 0 at a geometric rate when $t$ grows to infinity. Note that $\cA_t(1_B)=0$ for every $B\in B(\R^{N-1})$ such that $B\cap \partial \cC_t=\emptyset$. Next, it remains to show that there exist $c>0$ and $\rho\in(0,1)$ such that 
	\[\forall t>0, \quad  \| \cA_t(1_{\partial\mathcal{C}_t})\|_0 \le c \rho ^t.
\]
First, let us consider, for any $i\in\{1,\ldots,N-1\}$, the set $d_{i,t}:=\big\{ (y_1,\ldots,y_{N-1})\in \overline{\cC}_t \mid y_i=0\big\}$:
\begin{eqnarray*}
	\alpha_{k,\ell,t}(1_{d_{i,t}}) 
	& \le & \P_k\{L_t(i)=0\}  = \begin{cases} 0 & \text{ if $k= i$} \\
	\sum_{\ell }e^{t G_{i^ci^c}}(k,\ell) & \text{ if $k\neq i$.} \end{cases}
\end{eqnarray*}
 Thus,  $\max_i\| \cA_t(1_{d_{i,t}})\|_0 = O(e^{-rt})$ from (\ref{Perron-Frobenius}). 
Second let us denote $s_{t}:=\big\{ (y_1,\ldots,y_{N-1})\in \overline{\cC}_t \mid \sum_{j=1}^{N-1} y_j =t\big\}$. We can write for all $(k,\ell)\in\X^2$
\begin{eqnarray*}
	\alpha_{k,\ell,t}(1_{s_{t}}) & \le & \P_k\{L_t(N)=0,X_t=\ell\} \\
	& \le & \P_k\{L_t(N)=0\}  = \begin{cases} 0 & \text{ if $k= N$} \\
	\sum_{\ell }e^{t G_{N^cN^c}}(k,\ell) & \text{ if $k\neq N$.} \end{cases}
\end{eqnarray*}
Therefore, $ \|\cA_t(1_{s_{t}})\|_0=O(e^{-r t})$.
Combining the previous estimates, we obtain that there exist $c>0$ and $\rho\in(0,1)$ such that 
\[ \forall t >0, \quad \|\cA_t(1_{\R^{N-1}})\|_0 = \| \cA_t(1_{\partial\mathcal{C}_t})\|_0 \le c \rho^t.
\]
It follows that 
there exist  $c>0$ and $\rho\in(0,1)$ such that
\[  \forall t >0, \quad \|\cM_t(1_{\R^{N-1}})\|_0 = \| \cA_t(1_{\R^{N-1}})\|_0 \le c \rho^t.
\]
where $\cM_t$ is the matrix associated with the singular part of the probability distribution of $(X_t,Y'_t)$.

In a second step, we prove that for every $t_0>0$
$$ \sup_{t\in[t_0,2t_0)} \| \widehat \Psi_{t}(\zeta)\|_0 \longrightarrow 0\quad \text{when}\ \|\zeta\|\r+\infty .$$
Indeed, for every $t>0$, for every $(k,\ell)\in\X^2$, the Fourier transform $\widehat \psi_{k,\ell,t}$ of $\psi_{k,\ell,t}$ has the following form (if $N=2$ consider   only the first integral) 
	\[ \widehat \psi_{k,\ell,t}(\zeta) = \int_0^t \int_0^{t-y_1}\cdots \int_0^{t-\sum_{j\le N-2}y_j} \psi_{k,\ell,t}(y) e^{i\zeta_{N-1} y_{N-1}}dy_{N-1} \ldots dy_2\,dy_1. 
\]
Using an integration by part, we obtain that for every $\zeta_{N-1}\neq 0$
\begin{eqnarray*}
 \lefteqn{\int_0^{t-\sum_{j\le N-2}y_j} \psi_{k,\ell,t}(y) e^{i\zeta_{N-1} y_{N-1}} \, dy_{N-1} =} \\
 &  & \frac{1}{i\zeta_{N-1}} \left(\left[ \psi_{k,\ell,t}(y)e^{i \zeta_{N-1} y_{N-1}} \right]_{y_{N-1}=0}^{y_{N-1}=t-\sum_{j\le N-2}y_j} - \int_0^{t-\sum_{j\le N-2}y_j} \frac{\partial \psi_{k,\ell,t}}{\partial y_{N-1}}(y)\, e^{i y_{N-1}\zeta_{N-1}}\, dy_{N-1} \right)
\end{eqnarray*}
so that, we obtain from the bounds (\ref{Bornes_jointe_Psi})-(\ref{Borne_deriv_Psi}) that 
\begin{eqnarray*}
|\widehat\psi_{k,\ell,t}(\zeta)| & \le & \frac{1}{|\zeta_{N-1}|} \int_0^t \cdots \int_0^{t-\sum_{j\le N-3}y_j} \big(2a^{N-1} + 2a^N (t-\sum_{j\le N-2}y_j)\big) d y_{N-2}\ldots dy_{1} \\
& \le& \frac{1}{|\zeta_{N-1}|} \big(2a^{N-1} \frac{t^{N-2}}{(N-2)!} + 2 a^{N} \frac{t^{N-1}}{(N-1)!}\big).
\end{eqnarray*}
Finally, for every $t_0>0$
$$\sup_{t\in[t_0,2t_0)} \|\widehat\Psi_{t}(\zeta) \|_0 \leq \frac{2a^{N-1}(1+a)\max(1,t_0^{N})}{|\zeta_{N-1}|}.$$
For each $j\in\{1,\ldots,y_{N-2}\}$, using similar computations from an integration by parts with respect to the variable $y_j$, the same inequality holds with  $\zeta_j$ in place of $\zeta_{N-1}$. Therefore, we obtain that
\begin{equation*}
\forall \zeta\in\R^{N-1}\backslash \{0\}, \quad \sup_{t\in[t_0,2t_0)} \|\widehat\Psi_{t}(\zeta) \|_0 \leq \frac{2a^{N-1}(1+a)\max(1,t_0^{N})}{\|\zeta\|_0}.
\end{equation*}
and this quantity goes to 0 when $\|\zeta\|\r+\infty$.
Since $G_t(y)=\Psi_t(y+tm)$ we have $\widehat G_t (\zeta) = e^{-i\langle m',\zeta\rangle t}\widehat \Psi_t(\zeta)$ for any $\zeta \in\R^{N-1}$  
and 
	\[ \forall t>0, \ \forall \zeta \in\R^{N-1}, \quad \|\widehat G_t(\zeta) \|_0 = \| \widehat \Psi_t(\zeta)\|_0.
\]
It follows that 
\begin{equation*} 
\Gamma(\zeta)\equiv\Gamma_{t_0}(\zeta) :=  \sup_{t\in[t_0,2t_0)} \| \widehat G_{t}(\zeta)\|_0 \longrightarrow 0\quad \text{when}\ \|\zeta\|_0\r+\infty.
\end{equation*}

%

\section{A uniform LLT with respect to transition matrix $P$} \label{sec-unif}
Let $\cP$ denote the set of irreducible and aperiodic stochastic $N\times N$-matrices.  The topology in $\cP$ is  associated (for instance) with the distance $d(P,P') := \|P-P'\|_\infty\, $ ($P,P'\in\cP$). Again $(X_t,Y_t)_{t\in \T}$ is a centered MAP with state space $\X \times \R^d$, where  $\X:=\{1,\ldots,N\}$, and $\{P_t\}_{t\in\T}$ denotes the transition semi-group of the Markov process $\{X_t\}_{t\in\T}$. In this section, the stochastic matrix  $P:=P_1$ is assumed to belong to a compact subset $\cP_0$ of $\cP$, and we give assumptions for the LLT of Theorem~\ref{main-id} to hold uniformly in $P\in\cP_0$. 

For every $k\in\X$ the underlying probability measure $\P_k$ and the associated expectation $\E_k$ depend on $P$. To keep in mind this dependence, they are  denoted by $\P_k^{P}$ and $\E_k^{P}$ respectively. Similarly we use the notations $\Sigma^P$, $\cY_{t}^{P}$, $\cG_t^{P}$, $\cM_t^{P}$ and $G_t^{P}$ for the  covariance matrix of Theorem~\ref{lem-clt} and the matrices in (\ref{def-meas-mat})-(\ref{def-Gt-density}) respectively. Let $\mathfrak{M}$ denote the space of the probability measures on  $\R^d$ equipped with the total variation distance $d_{TV}$. 

Let $\cP_0$ be any compact subset of $\cP$. Let us consider the following assumptions: 
\begin{assu}: \label{U1}
For any $(k,\ell)\in \X^2$, the map $P\mapsto \cY_{k,\ell,1}^{P}$ is continuous from $(\cP_0,d)$ into $(\mathfrak{M},d_{TV})$. 
\end{assu}
\begin{assu}: \label{U4} 
There exist positive constants $\alpha$ and $\beta$ such that 
\begin{equation} \label{alpha-beta} 
\forall P\in\cP_0,\ \forall \zeta\in\R^d,\quad \alpha\, \|\zeta\|^2 \leq \langle \zeta,\Sigma^P \zeta \rangle \leq \beta\, \|\zeta\|^2. 
\end{equation} 
\end{assu}
\begin{assu}: \label{U3-global} 
The conditions $\M{3}$, \emph{\textbf{(AC\ref{AS2})}} and \emph{\textbf{(AC\ref{AS0})}} hold uniformly in $P\in \cP_0$.
\end{assu}
\begin{theo} \label{main-unif}
Under Assumptions~(U\ref{U1})-(U\ref{U3-global}), for every $k\in \X$, the density $f_{k,t}^{P}(\cdot)$ of the a.c.~part of the probability distribution of $t^{-1/2}Y_t$ under $\P_k^{P}$ satisfies the following asymptotic property when $t\r+\infty$: 
$$\sup_{P\in{\cal P}_0} \sup_{y\in\R^d}\big|f_{k,t}^{P}(y) - \eta_{_{\Sigma^P}}(y)\big| = O(t^{-1/2}) + O\big(\sup_{P\in{\cal P}_0}\sup_
{y\notin {\cal D}_t}\eta_{_{\Sigma^P}}(t^{-1/2}y) \big).$$
\end{theo}

Assumptions~(U\ref{U3-global}) read as follows: 
\begin{assu3}: \label{U3} 
The r.v.~$\{Y_v\}_{v\in(0,1]\cap\T}$ satisfies the following  moment condition: 
\begin{equation} \label{moment-alpha-bis} 
M:= \sup_{P\in{\cal P}_0} \max_{k\in\X}  \sup_{v\in(0,1]\cap\T} \E_k^{P}[\|Y_v\|^3] < \infty.  
\end{equation}
\end{assu3}
\begin{assu3}: \label{U2}
There exist $c>0$ and $\rho\in(0,1)$ such that   
\begin{equation} \label{masse-id-unif}
\forall t >0,\quad  \sup_{P\in{\cal P}_0} \| \cM_{t}^{P}(1_{\R^d}) \|_{0}  \leq c\rho^t 
\end{equation}
and there exists $t_0>0$ such that 
\begin{subequations}
\begin{equation} \label{cond-c-rho-unif}
\rho^{t_0}\max(2,cN)\leq 1/4 
\end{equation}
\begin{equation} \label{fourier-id-unif}
\Gamma_{t_0}(\zeta) :=  \sup_{P\in{\cal P}_0} \sup_{w\in[t_0,2t_0)} \| \widehat {G_{w}^{P}}(\zeta)\|_0 \longrightarrow 0\quad \text{when}\ \|\zeta\|\r+\infty.
\end{equation}
\end{subequations}
\end{assu3}
\begin{assu3}: \label{U2-cas1}
For any $t>0$, there exists an open convex subset $\cD_t$ of $\R^d$ such that $G_t$ vanishes on $\R^d\setminus\overline{\cD}_t$, where $\overline{\cD}_t$ denotes the adherence of $\cD_t$. Moreover $G_t$ is continuous on $\overline{\cD}_t$ and differentiable on $\cD_t$, with in addition 
\begin{subequations}
\begin{gather} 
\sup_{P\in{\cal P}_0} \sup_{t>0} \sup_{y\in \overline{{\cal D}}_t} \| G_t^{P}(y)\|_0 <\infty \label{Bornes_G-unif}\\ 
 \sup_{P\in{\cal P}_0} \sup_{y\in \partial {\cal D}_t}  \|G_t^{P}(y)\|_0 = O\big(\frac{1}{t}\big) \quad \text{where}\ \ \partial {\cal D}_t := \overline{{\cal D}}_t\setminus \cD_t \label{Bornes_G-front-unif} \\
 j=1,\ldots,d:\quad  \sup_{P\in{\cal P}_0}\sup_{t>0}\sup_{y\in {\cal D}_t} \big\| \frac{\partial G_t^{P}}{\partial y_j}(y)\big\|_0 <\infty. \label{Bornes_G-dif-unif} 
\end{gather}
\end{subequations}
\end{assu3}

\begin{proof}{}
The proof of Theorem~\ref{main-unif} borrows the same way as for Theorem~\ref{main-id}. What we only have to do is to prove that the bounds in Lemmas~\ref{lem-dec-vp}-\ref{lem-hat-g} are uniform in $P\in\cP_0$ under Assumptions~(U\ref{U1})-(U\ref{U4}). For Lemma~\ref{lem-hat-g}, this is obvious by using (U3-\ref{U2}). The proof of Theorem~\ref{main-unif} will be complete if we establish the uniform version of Lemma~\ref{lem-dec-vp} and \ref{lem-non-ari}. This is done below.  
\end{proof}
\begin{lem} \label{lem-cont-enP}
Assume that Condition~(U\ref{U1}) holds and that, for some $\alpha>0$, 
$$M_\alpha:= \sup_{P\in{\cal P}_0} \max_{k\in\X}  \sup_{v\in(0,1]\cap\T} \E_k^{P}[\|Y_v\|^\alpha] < \infty.$$ 
Then the map $(P,\zeta)\mapsto \widehat{\cY_1^{P}}(\zeta)$ is continuous  from $\cP_0\times\R^d$ into $\cM_N(\C)$.
\end{lem}
\begin{proof}{}
We may suppose that $\alpha\in(0,1]$. Let $(k,\ell)\in \X^2$, $(P,P')\in\cP_0^2$ and $(\zeta,\zeta')\in\R^d\times\R^d$. To simplify we write $\cY$ and $\cY'$ for $\cY_{k,\ell,1}^{P}$ and $\cY_{k,\ell,1}^{P'}$.  Then 
\begin{eqnarray*}
\big| \big(\widehat{\cY_1^{P}}(\zeta)\big)_{k,\ell} - \big(\widehat{\cY_1^{P'}}(\zeta')\big)_{k,\ell} \big| &=& \big| \int_{\R^d} e^{i \langle \zeta , y \rangle} \cY(dy) - \int_{\R^d} e^{i \langle \zeta' , y \rangle} \cY'(dy)\big| \\
&\leq& \int_{\R^d} \big| e^{i \langle \zeta , y \rangle} - e^{i \langle \zeta' , y \rangle}\big| \, \cY(dy) \\
&\ & \qquad +\ \  \big|\int_{\R^d} e^{i \langle \zeta' , y \rangle} \cY(dy) - \int_{\R^d} e^{i \langle \zeta' , y \rangle} \cY'(dy)\big| \\
&\leq& 2\, \|\zeta-\zeta'\|^\alpha\, \int_{\R^d} \|y\|^\alpha\, \cY(dy) + d_{TV}(\cY,\cY')\\
&\leq&  2 M_{\alpha}\, \|\zeta-\zeta'\|^{\alpha} + d_{TV}(\cY,\cY')
\end{eqnarray*}
(use $|e^{iu}-1| \leq 2|u|^\alpha$, $u\in\R$, and  the Cauchy-Schwarz inequality to obtain the second inequality above). The desired continuity property then follows from (U\ref{U1}). 
\end{proof} 

From now on, we sometimes omit the notational exponent $P$. The uniformity in Lemma~\ref{lem-dec-vp} is obtained as follows. Recall that, for any $P\in\cP_0$ fixed, Formula~(\ref{exp-val-pert}) with $t=n\in\N$ follows from the standard perturbation theory.  Similarly, using Lemma~\ref{lem-cont-enP}, Formula~(\ref{exp-val-pert}) with $t=n\in\N$ can be obtained for every $P\in\cU_0$ and for all $\zeta\in\R^d$ such that $\|\zeta\| \leq \delta$, with $\delta>0$ independent from $P_0\in\cP_0$ since $\cP_0$ is compact. Moreover the associated functions $\lambda(\cdot)$, $L_{k,t}(\cdot)$ and $R_{k,t}(\cdot)$ in (\ref{exp-val-pert}) (depending on $P$) satisfy the properties (\ref{exp-val-pert0})-(\ref{exp-val-pert2}) in a uniform way in $P\in\cP_0$ from (U3-\ref{U3}). Note that Condition~(U\ref{U4}) is useful to obtain  (\ref{exp-val-pert0-bis}) uniformly in $P\in\cP_0$. The passage from the discrete-time case to the continuous-time again follows from \cite[Prop.~4.4]{FerHerLed12} since the derivatives of $\zeta\mapsto\widehat{\cY_v}(\zeta)$ are uniformly bounded  in $(P,v)\in\cP_0\times(0,1]$ from (U3-\ref{U3}). 

From $\phi_{k,t}(\zeta) = e_k \widehat\cY_1(\zeta)^{\lfloor t \rfloor} \widehat\cY_v(\zeta) \mathbf{1}^{\top}$ 
with $v:=t-\lfloor t \rfloor\in[0,1)$
 and $\|\widehat{\cY_v}(\zeta){\bf 1}\|_\infty\leq\|\widehat{\cY_v}(0)\|_\infty\leq1$, the uniformity in Lemma~\ref{lem-non-ari}  follows from the following lemma. 
\begin{lem} \label{lem-NL-unif}
Let $\delta,A$ be such that $0<\delta<A$. Then there exist $D\equiv D(\delta,A)>0$ and $\tau\equiv \tau(\delta,A)\in(0,1)$ such that 
\begin{equation} \label{non-ari-spec} 
\forall n\in\N,\quad \sup_{ P\in{\cal P}_0}\sup_{\delta\leq\|\zeta\|\leq A} \|\widehat{\cY_1^{P}}(\zeta)^n\|_\infty \leq D\, \tau^n.
\end{equation}
\end{lem}
\begin{proof}{}
The spectral radius of any matrix $T\in\cM_N(\C)$ is denoted by $r(T)$. Suppose that \begin{equation} \label{ray-spec} 
\rho_0 := \sup_{P\in{\cal P}_0}\sup_{\delta\leq\|\zeta\|\leq A} r\big(\widehat{\cY_1^{P}}(\zeta)\big)<1.
\end{equation}
Then (\ref{non-ari-spec}) holds. Indeed, consider any $\tau\in(\rho_0,1)$ and denote by $\Gamma_\tau$ the oriented circle in $\C$ centered at $0$ with radius $\tau$. Let $I$ denote the identity $N\times N$-matrix.  Property~(\ref{ray-spec}) and standard spectral calculus then give 
\begin{eqnarray*}
\sup_{ P\in{\cal P}_0}\sup_{\delta\leq\|\zeta\|\leq A} \|\widehat{\cY_1^{P}}(\zeta)^n\|_\infty &\leq& 
\sup_{ P\in{\cal P}_0}\sup_{\delta\leq\|\zeta\|\leq A} \big\|\frac{1}{2i\pi} \oint_{\Gamma_\tau} z^n\, \big(zI-\widehat{\cY_1^{P}}(\zeta)\big)^{-1}\, dz\big\|_\infty \\
&\leq& \tau^{n+1}\sup_{ P\in{\cal P}_0}\sup_{\delta\leq\|\zeta\|\leq A}\|(zI-\widehat{\cY_1^{P}}(\zeta))^{-1}\|_\infty.
\end{eqnarray*}
This proves (\ref{non-ari-spec}) since the last bound is finite from the continuity of $(P,\zeta)\mapsto \widehat{\cY_1^{P}}(\zeta)$  on the compact set $\cP_0\times\{\zeta\in\R^d : \delta\leq\|\zeta\|\leq A\}$ (Lemma~\ref{lem-cont-enP}).  

It remains to prove (\ref{ray-spec}). Assume that $\rho_0\geq1$. Then $\rho_0=1$ since $r(\widehat{\cY_1}(\zeta)) \leq \|\widehat{\cY_1}(\zeta)\|_\infty\leq\|\widehat{\cY_1}(0)\|_\infty\leq1$. Thus there exists some sequences $(P_n)_n\in\cP_0^{\N}$ and $(\zeta_n)_n\in(\R^d)^{\N}$ satisfying $\delta\leq\|\zeta_n\|\leq A$ such that 
$$\lim_n r\big(\widehat{\cY_1^{P_n}}(\zeta_n)\big) = 1$$
By compactness one may suppose that $(P_n)_n$ and  $(\zeta_n)_n$  respectively converge to some $P_\infty\in\cP_0$ and some  $\zeta_\infty\in\R^d$ such that $\delta\leq\|\zeta_\infty\|\leq A$. From  Lemma~\ref{lem-cont-enP} we obtain that 
$$\lim_n\widehat{\cY_1^{P_n}}(\zeta_n) = \widehat{\cY_1^{P_\infty}}(\zeta_\infty).$$ 
Below, $\cY_1^{P_\infty}(\zeta_\infty)$ is simply denoted by $\cY_1(\zeta_\infty)$. From the upper semi-continuity of the map $T\mapsto r(T)$ on $\cM_N(\C)$ and from $r(\widehat{\cY_1}(\zeta_\infty))\leq 1$, it follows that $r(\widehat{\cY_1}(\zeta_\infty))=1$. Write $S:=\widehat{\cY_1}(\zeta_\infty)$ to simplify. Then $S$ admits an eigenvalue $\lambda$ of modulus one. Let $f=(f(k))_{k\in\X}\in\C^N$ be an associated nonzero eigenvector. We have 
$$\forall k\in\X,\quad |\lambda f(k)| = |f(k)| \leq |(S f)(k)| \leq (\widehat{\cY_1}(0)|f|)(k) = (P|f|)(k),$$ 
where $|f|:=(|f(k)|)_{k\in\X}$. Recall that the $P_\infty$-invariant probability measure $\pi=(\pi(\ell))_{\ell\, \in\X}$ is such that $\forall \ell\in\X,\ \pi(\ell)>0$ since $P_\infty$ is irreducible. From $\pi(P_\infty|f|)=\pi(|f|)$ and the positivity of $P_\infty|f|-|f|$, it follows that $P_\infty|f|=|f|$. Thus $|f|=c{\bf 1}$ with some constant $c$. We may assume that $|f|={\bf 1}$. Equality $\lambda f = S f$ rewrites as 
$$\forall k\in\X,\quad  \lambda f(k) = \sum_{\ell=1}^N f(\ell)\, \widehat{\cY_1}(\zeta_\infty)_{k,\ell} = \sum_{\ell=1}^N f(\ell)\,  \E_{k}^{P_{\infty}}\big[1_{\{X_1=\ell\}}\, e^{i \langle \zeta_\infty , Y_1 \rangle}\big] = \E_{k}^{P_{\infty}}\big[f(X_1)\, e^{i \langle \zeta_\infty , Y_1 \rangle}\big].$$
From $|f(k)|=1$ for every $k\in\X$ and from standard convexity arguments, we obtain 
$$\forall k\in\X,\quad \lambda f(k) = f(X_1)\, e^{i \langle \zeta_\infty , Y_1 \rangle}\ \ \ \P_k^{P_\infty}-\text{a.s.}.$$
Now writing $\lambda = e^{ib}$ with $b\in\R$ and $f(\ell)=e^{ig(\ell)}$ for every $\ell\in\X$, one can deduce from the previous equality that 
$$\forall k\in\X,\quad \langle \zeta_\infty , Y_1 \rangle + g(X_1) - g(k) \in b+2\pi\Z \ \ \ \P_k^{P_\infty}-\text{a.s.}.$$
Define $a:=\frac{b}{\|\zeta_\infty\|^2}\zeta_\infty\in\R^d$ and $\theta : \X\r\R^d$ by $\theta(\ell):= \frac{g(\ell)}{\|\zeta_\infty\|^2}\zeta_\infty$. Consider the following closed subgroup $H := \frac{2\pi\Z}{\|\zeta_\infty\|^2}\, \zeta_\infty\oplus(\R\cdot\zeta_\infty)^{\perp}$ in $\R^d$. Then the previous property is equivalent to 
$$\forall k\in\X,\quad Y_1 + \theta(X_1) - \theta(k) \in a+H \ \ \ \P_k^{P_\infty}-\text{a.s.}.$$
But Assumption~(U3-\ref{U2}) obviously implies that $P_\infty$ satisfies \textbf{(AC\ref{AS2})}, so that the last property is impossible as seen in the proof of Lemma~\ref{lem-non-ari}. Property~(\ref{ray-spec}) is proved.  
\end{proof}

\appendix

\section{Proof of Lemma~\ref{Lem52}} \label{B}

Assume that  
$$\forall B'\in B(\R^{N-1}),\quad \P_k\big\{X_t=\ell, Y'_t\in B'\big\} = \int_{\R^{N-1}} g'(y)\, dy + \mu'(A)$$
avec $\mu'\perp \ell_{N-1}$. Let $B\in B(H)$. Then  
\begin{eqnarray*}
\P_k\big\{X_t=\ell, Y_t\in B\big\}& =& \P_k\big\{X_t=\ell, \Lambda(Y_t)\in \Lambda(B)\big\} \\
&=& \int_{\Lambda(B)} g'(y)\, dy + \mu'(\Lambda(B))\\
&=& \int_{\R^{N-1}} 1_B(\Lambda^{-1}(y))\, g'\big(\Lambda(\Lambda^{-1}y)\big)\, dy + \mu'(\Lambda(B)) \\
&=& \int_B (g'\circ\Lambda)(x)\, d\eta(x) + \mu'(\Lambda(B)).
\end{eqnarray*}
Let $d\mu$ be the image measure of $d\mu'$ under $\Lambda^{-1}$, that is : $\forall B\in B(H),\ \mu(B) := \mu'(\Lambda(B))$. Then we have $\mu\perp \eta$. Indeed, we know that there exist two disjoint sets $E', F'\in B(\R^{N-1})$ such that  
$$\forall B'\in B(\R^{N-1}),\quad \ell_{N-1}(B') = \ell_{N-1}(B'\cap E') \quad \text{and} \quad \mu'(B') = \mu'(B'\cap F').$$
Then, introducing $E:=\Lambda^{-1}(E')$ and $F:=\Lambda^{-1}(F')$, we clearly have for any $B\in B(H)$ 
$$\eta(B) = \ell_{N-1}(\Lambda(B)) = \ell_{N-1}\big(\Lambda(B)\cap E'\big) = \ell_{N-1}\big(\Lambda(B\cap E)\big) = \eta(B\cap E)$$
and 
$$\mu(B) = \mu'(\Lambda(B)) = \mu'\big(\Lambda(B)\cap F'\big) = \mu'\big(\Lambda(B\cap F)\big) = \mu(B\cap F),$$
so that $\eta$ and $\mu$ are supported by the two disjoint sets $E$ et $F$. 

\section{Proof of Lemma~\ref{Lem_Psi}} \label{annexeA}

Let us recall that the density $\psi_{k,\ell,t}$ is given by (\ref{densite_jointe}). We can be a little bit more precise on the properties of the coefficients invoked in (\ref{densite_jointe}). Indeed, we know from \cite{Ser99} that 
	\[  p_{k,\ell}(n+N,k_1,\ldots,k_{N-1}) := \P_k\{V_{n+N-1}^1=k_1+1,\ldots,V_{n+N-1}^{N-1}=k_{N-1}+1,Z_{n+N-1}=\ell\}
\]
with $V_n^i$ the number of visits to state $i$ at time $n$
\begin{equation} \label{nb_visites}
\forall i=1,\ldots,N-1\quad  V_n^i=\sum_{j=0}^n 1_{\{Z_j =i\}}.
\end{equation}
	Moreover, if $y=(y_1,\ldots,y_{N-1})$
	\begin{equation} \label{Coeffts}
	  x_{n;k_1,\ldots,k_{N-1}}^{t;y} :=\frac{n!}{k_1!\cdots k_{N-1}! (n-\sum_{j=1}^{N-1}k_j)!}\prod_{j=1}^{N-1} \left( \frac{y_j}{t}\right)^{k_j}\bigg(1-\frac{\sum_{j=1}^{N-1} y_j}{t}\bigg)^{n-\sum_{j=1}^{N-1} k_j} 
\end{equation}	
Recall that these coefficients satisfy the property~(\ref{binomial}).

We deduce (\ref{Bornes_jointe_Psi}) from $\sum_{\ell \in\X}p_{k,\ell}(n+N,k_1,\ldots,k_{N-1})\le 1$ and (\ref{binomial}):
\begin{eqnarray*} 
\forall y \in \R^{N-1}, \quad |\psi_{k,\ell,t}(y)|\le \sum_{j \in\X}|\psi_{k,j,t}(y)| \le 1_{\mathcal{C}_t}(y)\, a^{N-1} \sum_{n=0}^{\infty} e^{-a t} \frac{(a t)^n}{n!} \le a^{N-1}.
\end{eqnarray*}

Let us check (\ref{Borne_deriv_Psi}) for $j:=1$. We obtain 
\begin{eqnarray*}
	\lefteqn{\frac{\partial \psi_{k,\ell,t}}{\partial y_1}} \\
	& = & a^{N-1} \sum_{n=1}^{\infty} e^{-a t} \frac{(a t)^n}{n!} \frac{1}{t}\\
 &  & \times \bigg[ \sum_{\tiny \begin{array}{c}k_1\ge 1,\ldots, k_{N-1}\ge 0,\\ \sum_{j=1}^{N-1} k_j\le n\end{array}} n \, x_{n-1;k_1-1,\ldots,k_{N-1}}^{t;y}\,  p_{k,\ell}(n+N,k_1,\ldots,k_{N-1})  \\
 & & \qquad -  \sum_{\tiny \begin{array}{c}k_1\ge 0,\ldots, k_{N-1}\ge 0,\\ \sum_{j=1}^{N-1} k_j\le n-1\end{array}} n \, x_{n-1;k_1,\ldots,k_{N-1}}^{t;y}\,  p_{k,\ell}(n+N,k_1,\ldots,k_{N-1})\bigg] 
\end{eqnarray*}
so that 
\begin{eqnarray*}
	\lefteqn{\frac{\partial \psi_{k,\ell,t}}{\partial y_1} = a^{N} \sum_{n=1}^{\infty} e^{-a t} \frac{(a t)^{n-1}}{(n-1)!} }\\
 &  & \times \bigg[ \sum_{\tiny \begin{array}{c}k_1\ge 0,\ldots, k_{N-1}\ge 0,\\ \sum_{j=1}^{N-1} k_j\le n-1\end{array}}  \, x_{n-1;k_1,\ldots,k_{N-1}}^{t;y}\,  p_{k,\ell}(n+N,k_1+1,k_2,\ldots,k_{N-1})  \\
 & & \qquad -  \sum_{\tiny \begin{array}{c}k_1\ge 0,\ldots, k_{N-1}\ge 0,\\ \sum_{j=1}^{N-1} k_j\le n-1\end{array}} \, x_{n-1;k_1,\ldots,k_{N-1}}^{t;y}\,  p_{k,\ell}(n+N,k_1,\ldots,k_{N-1})\bigg] .
\end{eqnarray*}
Finally, 
\begin{eqnarray*}
	\lefteqn{\frac{\partial \psi_{k,\ell,t}}{\partial y_1}(y) = a^{N} \sum_{n=0}^{\infty} e^{-a t} \frac{(a t)^{n}}{n!}  \sum_{\tiny \begin{array}{c}k_1\ge 0,\ldots, k_{N-1}\ge 0,\\ \sum_{j=1}^{N-1} k_j\le n\end{array}}  \, x_{n;k_1,\ldots,k_{N-1}}^{t;y}\,}\\
	& &  \times \big[ p_{k,\ell}(n+N+1,k_1+1,k_2,\ldots,k_{N-1}) -  p_{k,\ell}(n+N+1,k_1,\ldots,k_{N-1})\big] .
\end{eqnarray*}
It follows that 
\begin{eqnarray*} 
\forall y \in \cC_t, \quad \big|\frac{\partial \psi_{k,\ell,t}}{\partial y_1}(y)\big|\le \sum_{\ell \in\X}\big|\frac{\partial \psi_{k,\ell,t}}{\partial y_1}(y)\big| \le a^{N} 2 \sum_{n=0}^{\infty} e^{-a t} \frac{(a t)^n}{n!} \le 2a^{N}.
\end{eqnarray*}
The same computation for each variable $y_j$ $j=2,\ldots,N-1$ gives the bound in (\ref{Borne_deriv_Psi}).

For any $t>0$, the definition of function $\Psi_t$ can be extended on the boundary $\partial \cC_t$ (see (\ref{frontier})) since the coefficients in $(\ref{Coeffts})$ are well defined on $\overline{\cC}_t$. 	
Moreover, since these coefficients as function of $y$ are continuous on $\overline{\cC}_t$ , it is easily seen that the extended version of $\Psi_t$ is continuous on $\overline{\cC}_t$ and is also denoted by $\Psi_t$ in the sequel. Next, we study the behaviour of this function on the boundary $\partial \cC_t$. 

We have  for any $y\in \big\{ (y_1,\ldots,y_{N-1})\in \overline{\cC}_t : y_i:=0\big\}$ that $x_{n;k_1,\ldots,k_{N-1}}^{t;y} = 0$ if $k_i\ge 1$, 
so that 
\begin{eqnarray*} 
\lefteqn{	\psi_{k,\ell,t}(y) := a^{N-1} \sum_{n=0}^{\infty} e^{-a t} \frac{(a t)^n}{n!}}\\
 &  & \times  \sum_{\tiny \begin{array}{c}k_1\ge 0,\ldots, k_i=0, \ldots,k_{N-1}\ge 0,\\ \sum_{j=1}^{N-1} k_j\le n\end{array}} \, x_{n;k_1,\ldots,0,\ldots, k_{N-1}}^{t;y}\,  p_{k,\ell}(n+N,k_1,\ldots,0,\ldots, k_{N-1}). \nonumber
\end{eqnarray*}
Moreover, using 
$p_{k,\ell}(n+N,k_1,\ldots,0,\ldots, k_{N-1}) \le \P_{k}\{V_{n+N-1}^i = 1\}$ and   
 (\ref{binomial}), we obtain  
\begin{subequations}
\begin{equation} \label{avantder_0}
\forall y\in \big\{ (y_1,\ldots,y_{N-1})\in \overline{\cC}_t \mid y_i:=0\big\}, \quad | \psi_{k,\ell,t}(y) | \le 
	a^{N-1} \sum_{n=0}^{\infty} e^{-a t} \frac{(a t)^n}{n!} \P_{k}\{V_{n+N-1}^i = 1\}. 
\end{equation}
Finally, for any $y\in \overline{\cC}_t$ such that $\langle y ,\mathbf{1}\rangle =t$, we have 
$x_{n;k_1,\ldots,k_{N-1}}^{t;y}  = 0$ if $\sum_{j=1}^{N-1} k_j < n$, 
so that 
\begin{eqnarray*} 
\lefteqn{	\psi_{k,\ell,t}(y) := a^{N-1} \sum_{n=0}^{\infty} e^{-a t} \frac{(a t)^n}{n!}}\\
 &  & \times  \sum_{\tiny \begin{array}{c}k_1\ge 0,\ldots, k_i \ge 0, \ldots,k_{N-1}\ge 0,\\ \sum_{j=1}^{N-1} k_j = n\end{array}} \, x_{n;k_1,\ldots, k_{N-1}}^{t;y}\,  p_{k,\ell}(n+N,k_1,\ldots, k_{N-1}). \nonumber
\end{eqnarray*}
Since $n=\sum_{j=1}^{N-1} k_j$, we have $\sum_{j=1}^{N-1} (k_j+1) = n+N-1$ and for any $(k,\ell) \in\X^2$,
\begin{equation*} 
p_{k,\ell}(n+N,k_1,\ldots,k_{N-1}) \le  \P_{k}\{V_{n+N-1}^N = 0\}. 
\end{equation*}
It allows us to write that 
\begin{equation} \label{avantder_t}
 y\in \bigg\{ (y_1,\ldots,y_{N-1})\in \overline{\cC}_t \mid \sum_{j=1}^{N-1} y_j =t\bigg\}, \quad 
 | \psi_{k,\ell,t}(y) | \le 
	a^{N-1} \sum_{n=0}^{\infty} e^{-a t} \frac{(a t)^n}{n!} \P_{k}\{V_{n+N-1}^N = 0\}. 
\end{equation}
\end{subequations}
We deduce from the next Lemma~\ref{nbvisits} and (\ref{avantder_0})-(\ref{avantder_t}) that there is  $\rho\in(0,1)$ such that 
\begin{equation*} 
	\forall y\in\partial \cC_t, \quad \|\Psi_t(y)\|_0 = O\big(e^{at(\rho-1)}(1+at)\big). 
\end{equation*}	

\begin{lem} \label{nbvisits} For a discrete time finite Markov chain $\{Z_n\}_{n\in\N}$ with transition matrix $\widetilde P$, let $V_n^i$ be the number of visits to state $i$ with respect to $n$ transitions of the Markov chain (see \emph{(\ref{nb_visites})}). Let us introduce the $(N-1)\times (N-1)$-matrix $\widehat P_{i^ci^c}:=(\widetilde P(k,\ell))_{k,\ell\in \X \backslash \{i\}}$ and assume that this matrix is irreducible and aperiodic. Let $0<\rho_i<1$ be the Perron-Frobenius eigenvalue of $\widehat P_{i^ci^c}$. We have the following geometric estimates of the probabilities in \emph{(\ref{avantder_0})-(\ref{avantder_t}) }:
\begin{enumerate}
	\item $\P_{i}\{V_{n}^i = 0\}=0$ and $\P_{k}\{V_{n}^i = 0\}=O({\rho_i}^n)$	
	\item $\P_{k}\{V_{n}^i = 1\} =O({\rho_i}^n + n {\rho_i}^{n-1})$ for $k\neq i$ and $\P_{i}\{V_{n}^i = 1\} =O({\rho_i}^n)$ 
\end{enumerate}
\end{lem}
\begin{rem} Note that the irreducibility of the sub-generators $G_{i^ci^c}$ allows us to derive exponential rate of convergence of $\|\Psi_t(y)\|_0$ on $\partial C_t$ when $t$ growths to infinity, given that only a rate in $1/t$ is required in Condition~\emph{\textbf{(AC2)}}.
\end{rem}
\begin{proof}{}
Since $\widehat P_{\{i\}^c\{i\}^c}$ is irreducible and aperiodic,  $\| \widehat P_{\{i\}^c\{i\}^c}\mathbf{1}^{\top} \|_{\infty}\| \le \|\widehat P_{\{i\}^c\{i\}^c} \|_{\infty} \le C \rho_i$ for some constant $C$ and $\rho_i\in(0,1)$ is the Perron-Frobenius eigenvalue of  $\widehat P_{\{i\}^c\{i\}^c}$.

For the first  assertion, note that $\P_{i}\{V_{n}^i = 0\}=0$ since $V_n^{i} \ge 1$ and for $k\neq i$
\begin{eqnarray}
\forall n\ge 0, \quad	\P_k\{V_{n}=0\} & = & \P_k\{\sum_{\ell=1}^{n} 1_{\{Z_\ell=i\}}=0\}  =   \P_k\{Z_1\neq i,Z_2\neq i, \ldots, Z_{n}\neq i\} \nonumber\\
	& = & \sum_{\ell\in\X}(\widehat P_{\{i\}^c\{i\}^c})^{n}(k,\ell) = e_k (\widehat P_{\{i\}^c\{i\}^c})^{n} \mathbf{1}^{\top}. \label{PiVn}
\end{eqnarray}
Then, it follows that $\P_{k}\{V_{n}^i = 0\}=O({\rho_i}^n)$.

For the second assertion, note that $\P_{i}\{V_{n}^i = 1\}=\P_{i}\{\sum_{j=1}^{n} 1_{\{Z_j=i\}}= 0\}= e_i (\widehat P_{\{i\}^c\{i\}^c})^{n} \mathbf{1}^{\top}$ so that, we deduce that $\P_{i}\{V_{n}^i = 1\}= O({\rho_i}^n)$.
For $k\neq i $, we can write the following renewal equation 
\begin{eqnarray*}
\forall n \ge 0, \quad 	\P_k\{V_{n+1}=1\} & = & \sum_{j\in\X} P(k,j)\P_j\{V_n=1\} \\
& = & \sum_{j\neq i} P(k,j)\P_j\{V_n=1\} + P(k,i) \P_i\{V_n=1\} \\
\end{eqnarray*}
with $ \P_j\{V_0=1\}=0$ if $j\neq i$ and $\P_i\{V_0=1\}=1$. Let us introduce the following column vector $v_i(n+1) = (\P_k\{V_{n+1}=1\})_{k\neq i}$. Then we have from the previous equation
\begin{eqnarray*}
	\forall n\ge 0, \quad v_i(n+1) &= & P_{\{i\}^c\{i\}^c} v_i(n) +  \P_i\{V_n=1\} P_{\{i\}^c\{i\}}, \qquad v_i(0)=0, \P_i\{V_0=1\}=1.
\end{eqnarray*}
Then, we can obtain the following representation of vector $v_i(n+1)$
\begin{eqnarray}
	v_i(n+1) = \sum_{k=0}^n \P_i\{V_k=1\} {P_{\{i\}^c\{i\}^c}}^{n-k} P_{\{i\}^c\{i\}}.
\end{eqnarray}
Note that $\P_i\{V_k=1\}$ is known from (\ref{PiVn}). We obtain 
\begin{eqnarray*}
	v_i(n+1) & = & {P_{\{i\}^c\{i\}^c}}^{n} P_{\{i\}^c\{i\}} + \sum_{k=1}^n \left( P_{\{i\}\{i\}^c}{P_{\{i\}^c\{i\}^c}}^{k-1} \, 1^{\top}\right) \times {P_{\{i\}^c\{i\}^c}}^{n-k} P_{\{i\}^c\{i\}}.
\end{eqnarray*}
The $k$th component of the vector is given by 
\begin{eqnarray*}
	\P_k\{V_{n+1}=1\} &= &{e_k}^{\top} v_i(n+1) \nonumber\\
	& = & {e_k}^{\top}{P_{\{i\}^c\{i\}^c}}^{n} P_{\{i\}^c\{i\}} \nonumber \\
	& & + \sum_{k=1}^n \left( P_{\{i\}\{i\}^c}{P_{\{i\}^c\{i\}^c}}^{k-1} \, 1^{\top}\right) \times {e_k}^{\top}{P_{\{i\}^c\{i\}^c}}^{n-k} P_{\{i\}^c\{i\}} \nonumber \\
	& \le &  {e_k}^{\top}{P_{\{i\}^c\{i\}^c}}^{n}1^{\top}  + \sum_{k=1}^n \left( P_{\{i\}\{i\}^c}{P_{\{i\}^c\{i\}^c}}^{k-1} \, 1^{\top}\right) \times {e_k}^{\top}{P_{\{i\}^c\{i\}^c}}^{n-k} \, 1^{\top}.
\end{eqnarray*}
Since $\|P_{\{i\}^c\{i\}^c}\|\le C \rho_i$, there exists $K$ such that 
\begin{equation*}
	\P_k\{V_{n+1}=1\}  \le   K({\rho_i}^{n} + \sum_{k=1}^n {\rho_i}^{k-1} \, {\rho_i}^{n-k}) = K({\rho_i}^{n} + n {\rho_i}^{n-1}) .
\end{equation*}
\vspace*{-1cm}
\end{proof}

\section{Linear transformation of MAPs}

Let $\{(X_t,Y_t)\}_{t\in \T}$ be an MAP  with state space $\X \times \R^d$, where  $\X:=\{1,\ldots,N\}$ and $\T:=\N$ or $\T:=[0,\infty)$. Recall from \cite{Asm03} that  $\{(X_t,Y_t)\}_{t\in \T}$ is a Markov process with  a transition semi-group, denoted by $\{Q_t\}_{t\in\T}$, which satisfies 
	\begin{equation} \label{Def_Add}
	 \forall (k,y)\in\X\times \R^d, \ \forall (\ell,B) \in \X\times B(\R^d),\quad 
	 Q_t(k,y;\{\ell\}\times B)=  Q_t(k,0;\{\ell\}\times B-y).
\end{equation}
Note that $Q_t(k,y;\{\ell\}\times B) := \P_k\{X_t=\ell,Y_t\in B\} = \cY_{k,\ell,t}(1_B)$. 
Let us consider any  linear transformation $T:\R^d \r \R^m$ and introduce the $\X\times \R^m$-valued process $\{(X_t,TY_t)\}_{t\in \T}$. 
\begin{lem} \label{lem_TMAP} The process $\{(X_t,TY_t)\}_{t\in \T}$ is an MAP with state space $\X\times \R^m$ and transition semi-group $Q^{(T)}_t$ defined by: $\forall (k,z)\in\X\times \R^m, \ \forall (\ell,B) \in \X\times B(\R^m)$,
	\begin{equation} \label{Def_TMAP}
	 Q^{(T)}_t(k,z;\{\ell\}\times B)=  Q_t(k,0;\{\ell\}\times T^{-1}(B-z)).
\end{equation}
\end{lem} 
Note that we only have to prove that $\{(X_t,TY_t)\}_{t\in \T}$ is a Markov process with transition semi-group defined by (\ref{Def_TMAP}), since the additivity property for $\{(X_t,TY_t)\}_{t\in \T}$ is clearly satisfied from (\ref{Def_TMAP}):
	\begin{eqnarray*}
	 Q^{(T)}_t(k,z;\{\ell\}\times B) & = & Q_t(k,0;\{\ell\}\times T^{-1}(B-z)) \\
	 & =: & Q^{(T)}_t(k,0;\{\ell\}\times B-z).
\end{eqnarray*}
\begin{proof}{}
Let $\F^{(X,Y)}_t:=\sigma(X_u,Y_u,\ u\le t)$ and $\F^{(X,TY)}_t:=\sigma(X_u,TY_u,\ u\le t)$  be the filtration generated by the processes $\{(X_t,Y_t)\}_{t\in \T}$ and $\{(X_t,TY_t)\}_{t\in \T}$ respectively. Since $\{(X_t,Y_t)\}_{t\in\T}$ is an MAP with transition semi-group $\{Q_t\}_{t\in\T}$, we have by definition for any bounded function $g$ on $\X\times \R^d$
\begin{eqnarray*}
	\E\left[ g(X_{t+s},Y_{t+s}) \mid \F_s^{(X,Y)}\right] & = & \sum_{\ell\in\X}\int_{\R^d} g(\ell,y_1) Q_t(X_s,Y_s;\{\ell\}\times dy_1)\\
	 & = & \sum_{\ell\in\X}\int_{\R^d} g(\ell,y_1+Y_s) Q_t(X_s,0;\{\ell\}\times dy_1).
\end{eqnarray*}
Using the tower rule and the last representation of the condition expectation, we obtain 
for any $\ell\in\X$ and $B\in B(\R^m)$
\begin{eqnarray*}
	\E\left[ 1_{\{\ell\}\times B}(X_{t+s},TY_{t+s}) \mid \F_s^{(X,TY)} \right] & = & 	\E\left[ \E\left[  1_{\{\ell\}\times B}(X_{t+s},TY_{t+s}) \mid \F_s^{(X,Y)}\right] \mid \F_s^{(X,TY)}\right]\\
	& = & \E\left[ 	\int_{\R^d}  1_{B}(T(y_1+Y_s)) Q_t(X_s,0;\{\ell\}\times dy_1)	   \mid \F_s^{(X,TY)}\right] \\
		& = & \E\left[ \int_{\R^d}  1_{B}(Ty_1+TY_s) Q_t(X_s,0;\{\ell\}\times dy_1)	   \mid \F_s^{(X,TY)}\right] \\
		& = & \int_{\R^d} 1_{B}(Ty_1+TY_s) Q_t(X_s,0;\{\ell\}\times dy_1) \\
				& = & \int_{\R^d} 1_{T^{-1}(B-TY_s)}(y_1) Q_t(X_s,0;\{\ell\}\times dy_1) \\
								& = & Q_t(X_s,0;\{\ell\}\times T^{-1}(B-TY_s)).	 
\end{eqnarray*}
Then $\{(X_t,TY_t)\}_{t\in \T}$ is a Markov process with transition semi-group given by (\ref{Def_TMAP}).
\end{proof}

\bibliographystyle{alpha}

\end{document}